\documentclass[11pt, english]{amsart}

\usepackage{amsmath,amssymb,enumerate}

\usepackage[T1]{fontenc}
\usepackage[all]{xy}

\usepackage{babel}
\usepackage{amstext}
\usepackage{amsmath}
\usepackage{amsfonts}
\usepackage{latexsym}
\usepackage{ifthen}

\usepackage{xypic}
\xyoption{all}
\newcommand{\sK}{{\mathcal K}}

\newtheorem{lemma1}{}[section]

\newenvironment{lemma}{\begin{lemma1}{\bf Lemma.}}{\end{lemma1}}

\newenvironment{theorem}{\begin{lemma1}{\bf Theorem.}}{\end{lemma1}}
\newenvironment{proposition}{\begin{lemma1}{\bf Proposition.}}{\end{lemma1}}
\newenvironment{corollary}{\begin{lemma1}{\bf Corollary.}}{\end{lemma1}}
\newenvironment{remark}{\begin{lemma1}{\bf Remark.}\rm}{\end{lemma1}}

\newenvironment{definition}{\begin{lemma1}{\bf Definition.}}{\end{lemma1}}
\newenvironment{notation}{\begin{lemma1}{\bf Notation.}}{\end{lemma1}}

\newenvironment{remark*}{{\bf Remark.}}{}
\newenvironment{example*}{{\bf Example.}}{}
\newenvironment{assumption*}{{\bf Assumption.}}{}

\newcommand{\R}{\ensuremath{\mathbb{R}}}
\newcommand{\Q}{\ensuremath{\mathbb{Q}}}
\newcommand{\Z}{\ensuremath{\mathbb{Z}}}
\newcommand{\C}{\ensuremath{\mathbb{C}}}
\newcommand{\N}{\ensuremath{\mathbb{N}}}
\newcommand{\PP}{\ensuremath{\mathbb{P}}}

\newcommand{\holom}[3]{\ensuremath{#1:#2  \rightarrow #3}}
\newcommand{\fibre}[2]{\ensuremath{#1^{-1} (#2)}}

\makeatletter
\ifnum\@ptsize=0  \fi \ifnum\@ptsize=2  \fi 

\newcommand\sA{{\mathcal A}}

\newcommand\sF{{\mathcal F}}

\newcommand\sH{{\mathcal H}}
\newcommand\sI{{\mathcal I}}

\newcommand\sO{{\mathcal O}}

\newcommand\sC{{\mathcal C}}
\newcommand\bR{{\mathbb R}}

\newcommand\bC{{\mathbb C}}

\newcommand\sD{{\mathcal D}}

\newcommand\pp{{\overline \partial}}

\DeclareMathOperator*{\pic}{Pic}
\DeclareMathOperator*{\sing}{sing}

\DeclareMathOperator*{\nons}{nons}

\newcommand{\Chow}[1]{\ensuremath{\mbox{\rm Chow}(#1)}}

\DeclareMathOperator*{\supp}{Supp}

\newcommand{\NEX}{\overline{\mbox{NE}}(X)}
\newcommand{\NAX}{\overline{\mbox{NA}}(X)}
\newcommand{\NE}[1]{ \ensuremath{ \overline { \mbox{NE} }(#1)} }
\newcommand{\NA}[1]{ \ensuremath{ \overline { \mbox{NA} }(#1)} }

\DeclareMathOperator*{\NS}{NS}

\title{Minimal Models for K\"ahler threefolds} 
\date{April 19 2015}


\author{Andreas H\"oring}
\author{Thomas Peternell}

\address{Andreas H\"oring, Laboratoire de Math{\'e}matiques J.A. Dieudonn{\'e},
UMR 7351 CNRS, Universit{\'e} de Nice Sophia-Antipolis, 06108 Nice Cedex 02, France        
}
\email{hoering@unice.fr}

\address{Thomas Peternell, Mathematisches Institut, Universit\"at Bayreuth, 95440 Bayreuth, 
Germany}

\email{thomas.peternell@uni-bayreuth.de}

\begin{document}

\begin{abstract} 
Let $X$ be a compact K\"ahler threefold that is not uniruled. 
We prove that $X$ has a minimal model.
\end{abstract}

\maketitle

\section{Introduction}

The minimal model program (MMP) is one of the cornerstones in the classification theory of  complex projective varieties. 
It is fully developed only in dimension 3, despite tremendous recent progress in higher dimensions, in particular by
\cite{BCHM10}. In the analytic K\"ahler situation the basic methods from the MMP, such as the base point free theorem, fail. Nevertheless it seems reasonable to expect that
the main results should be true also in this more general context. 

The goal of this paper is to establish the minimal model program for K\"ahler threefolds $X$ whose canonical bundle $K_X$ is pseudoeffective. To be more specific,  we prove the following result:

\begin{theorem} \label{theoremminimalmodel}
Let $X$ be a  normal $\Q$-factorial compact   K\"ahler threefold with at most terminal singularities.
Suppose that $K_X$ is pseudoeffective or, equivalently, that $X$ is not uniruled. Then $X$ has a minimal model, i.e. 
there exists a MMP 
$$
X \dashrightarrow X'
$$
such that $K_{X'}$ is nef.
\end{theorem}

In our context a complex space is said to be $\Q$-factorial if every Weil divisor is $\Q$-Cartier {\it and}  
a multiple $(K_X^{\otimes m})^{**}$  of  the canonical sheaf $K_X$ is locally free. 
If $X$ is a projective threefold, Theorem \ref{theoremminimalmodel} was established 
by the seminal work of Kawamata, Koll\'ar, Mori, Reid and Shokurov \cite{Mor79, Mor82, Rei83, Kaw84, Kaw84b, Kol84, Sho85, Mor88, KM92}.
Based on the deformation theory of rational curves on smooth threefolds \cite{Kol91, Kol96}, Campana and the second-named 
author started to investigate the existence of Mori contractions \cite{CP97, Pet98, Pet01}. Nakayama \cite{Nak02b} proved the existence of an (even good) minimal model
for a K\"ahler threefold of algebraic dimension $2.$

The equivalence of $X$ being non-uniruled and $K_X$ being pseudoeffective for a threefold $X$ is a remarkable result due to Brunella \cite{Bru06}\footnote{If $X$ is projective
this is known in arbitrary dimension by \cite{BDPP13}.}.  This settles in particular the problem of the existence $K_X$-negative 
rational curves in case $K_X$ is not 
pseudoeffective. If $K_X$ is pseudoeffective, but not nef, ad hoc methods provide $K_X$-negative rational curves, as we will show in this paper. 
One main step in the MMP is the construction of extremal contractions, which in the projective case follows from the base-point free theorem. In the K\"ahler case we need
a completely different approach. 

Restricting from now on to varieties $X$ with 
$K_X$ pseudoeffective, we
consider the 
divisorial Zariski decomposition \cite{Bou04}
$$
K_X = \sum_{j=1}^r \lambda_j S_j + N(K_X).
$$
Here $N(K_X)$ is an $\mathbb R-$divisor which is ``nef in codimension one''. 
If $K_X|_{S_j}$ is not pseudoeffective we use a (sub-)adjunction argument to show that $S_j$ is uniruled.
We can then prove that $K_X$ is not nef (in the sense of \cite{DPS94}) if and only if there
exists a curve $C \subset X$ such that $K_X \cdot C<0$. In Section \ref{sectionBB} we show how deformation theory
on the threefold $X$ and the (maybe singular) surfaces $S_j$ can be used to establish an analogue
of Mori's bend and break technique. As a consequence we derive  
the cone theorem for the Mori cone $\NEX$ (cf. Theorem \ref{theoremNE}).  

However the Mori cone $\NEX$ is not the right object to consider: even if we find a bimeromorphic morphism
$X \rightarrow Y$ contracting exactly the curves lying on some $K_X$-negative extremal ray in $\NEX$,
it is not clear that $Y$ is a K\"ahler space. However it had been observed in \cite{Pet98} that 
the K\"ahler condition could be preserved if we contract extremal rays in $\NAX$, the
cone generated by positive closed currents of bidimension $(1,1)$.
Based on the description of $\NAX$ by Demailly and P\v aun \cite{DP04}, we prove the following
cone theorem:

\begin{theorem} \label{theoremNA}
Let $X$ be a normal $\Q$-factorial compact K\"ahler threefold with at most terminal singularities
such that $K_X$ is pseudoeffective. 
Then there exists a countable family $(\Gamma_i)_{i \in I}$  of rational curves  on $X$
such that 
$$
0 < -K_X \cdot \Gamma_i \leq 4
$$
and
$$
\NAX = \NAX_{K_X \geq 0} + \sum_{i \in I} \R^+ [\Gamma_i] 
$$
\end{theorem}

We then proceed to prove the existence of contractions of an extremal ray $\R^+ [\Gamma_i]$. 
If the curves $C \subset X$ such that $[C] \in \R^+ [\Gamma_i]$ cover a divisor $S$, we can use
a generalisation of Grauert's criterion \cite{Gra62} by Ancona and Van Tan \cite{AV84}.

If the curves in the extremal ray cover only a 1-dimensional set $C$ (so the contraction, provided it exists, is small),
the problem is more subtle. We consider a nef supporting class $\alpha,$ i.e., $\alpha \cdot C = 0.$ It turns
out that $\alpha$ is actually big: $\alpha^3 > 0.$ Using results of Collins - Tosatti \cite{CT13} and Boucksom \cite{Bou04}, we establish a 
modification $\mu: \tilde X \to X$ and a K\"ahler class $\tilde \alpha $ on $\tilde X$ such that 
$$\mu^*(\alpha) = \tilde \alpha + E$$
with $E$ an effective $\mu-$antiample $\mu-$exceptional divisor. After possibly a small perturbation of $E$, 
we show that $E = \sum m_i D_i + R$ where the $D_i$ cover $\mu^{-1}(C)$ and $R$ is disjoint from $E$. 
Therefore the conormal bundle $N^*_{\sum m_i D_i}$ is $\mu-$ample, and by Grauert's criterion, there is 
a morphism $\psi: \tilde X \to Y$ contracting $\bigcup D_i$ to a point. Then $\psi$ factors through a map 
$\varphi: X \to Y$ which is the contraction of $C$ we are looking for.

In summary we have proven the
contraction theorem:

\begin{theorem} \label{theoremcontraction}
Let $X$ be a normal  $\Q$-factorial compact K\"ahler threefold with at most terminal singularities
such that $K_X$ is pseudoeffective. Let $\R^+ [\Gamma_i]$ be a $K_X$-negative extremal ray in $\NAX$.
Then the contraction of $\R^+ [\Gamma_i]$ exists in the K\"ahler category
(cf. Definition \ref{definitioncontraction}).
\end{theorem}

Since Mori's theorem \cite{Mor88} assures the existence of flips also in the analytic category, termination being known by
Shokurov \cite{Sho85}, we can now 
run the MMP and obtain Theorem \ref{theoremminimalmodel}. 

By \cite[Thm.0.3]{DP03} this also implies that
the non-vanishing conjecture holds for compact K\"ahler threefolds:

\begin{corollary}
Let $X$ be a  normal $\Q$-factorial compact (non-projective)   K\"ahler threefold with at most terminal singularities.
Then $X$ is uniruled if and only if $\kappa(X)=-\infty$.
\end{corollary}

Concerning abundance, the results of this paper, using \cite{Pet01} yield: 

\begin{corollary} \label{cor:abund} Let $X$ be a  normal $\Q$-factorial compact (non-projective)   K\"ahler threefold with at most terminal singularities.
Suppose that $K_X$ is nef. Then $mK_X$ is spanned for some positive $m$, unless (possibly) there is no positive-dimensional
subvariety through the very general point of $X$ and $X$ is not bimeromorphic to $T/G$ where $T$ is a torus and $G$ a finite group
acting on $T.$ 
\end{corollary}

In \cite{CHP14}, we show the abundance conjecture for K\"ahler threefolds in full generality; thus the additional assumption in Corollary \ref{cor:abund}
is superfluous.

We end the introduction with a discussion on uniruled non-algebraic K\"ahler threefolds. Here the global bimeromorphic structure is rather
simple; as usual, $a(X) $ denotes the algebraic dimension.  Since $X$ is uniruled and non-algebraic, the rational quotient (MRC-fibration) is an almost holomorphic 
map $f: X \dasharrow S$ over a  smooth non-algebraic surface $S$.  

\begin{theorem} \cite{CP00}
Let $X$ be a smooth compact K\"ahler threefold that is uniruled and not projective. Then $X$ falls in one of the following classes. 
\begin{enumerate} 
\item $a(X) = 0$ and there is an almost holomorphic $\mathbb P_1-$fibration $f: X \dasharrow S$ over a smooth surface $S$ with $a(S) = 0.$ 
\item $a(X) = 1$ and we are in one of the following subcases.
\begin{enumerate}
\item $X$ is bimeromorphic to $\mathbb P_1 \times F$ with $a(F) = 0.$
\item the algebraic reduction $X \to B$ is holomorphic onto a smooth curve $B $ and  the general fibre is of the form $\mathbb P(\sO \oplus L)$
with $L$ a topologically trivial line bundle over an elliptic curve which is not torsion. 
\end{enumerate}
\item $a(X) = 2 $, $X$ has a meromorphic elliptic fibration over an algebraic surface. 
\end{enumerate}
\end{theorem} 

By \cite{HP13b}, the MMP for uniruled threefolds exhibits a birational $X'$ which is  a Mori fibre space 
$\varphi: X' \to S$  over a non-algebraic surface $S$.

{\bf Acknowledgements.} 
We thank the Forschergruppe 790 ``Classification of algebraic surfaces and compact complex manifolds'' of the Deutsche
Forschungsgemeinschaft for
financial support. 
A. H\"oring was partially also supported by the A.N.R. project CLASS\footnote{ANR-10-JCJC-0111}.
This work was done when A. H\"oring was member of the Institut de Math\'ematiques de Jussieu (UPMC).
Last but not least we thank the referees for very helpful comments.

\section{Basic Notions} \label{sectionnotation}

We will use frequently standard terminology  
of the minimal model program (MMP) as explained in \cite{KM98} or \cite{Deb01}.
In particular we say that a normal complex space $X$ is $\Q$-factorial if for every Weil divisor $D$
there exists an integer $m \in \N$ such that $\sO_X(mD)$ is a locally free sheaf, i.e. $mD$ is a Cartier divisor.
Since however the canonical sheaf $K_X = \omega_X$ need not be a $\mathbb Q$-Weil divisor, we include in the 
definition of $\mathbb Q-$factoriality also the condition that there is a number $m \in \N$ such that the coherent sheaf
$$
(K_X^{\otimes m})^{**} = (\omega_X^{\otimes m})^{**}
$$
is locally free. We shall write 
$$  mK_X = (K_X^{\otimes m})^{**}Ê$$
for short.

Somewhat abusively we will denote the cycle space or Barlet space of a complex space $X$ by $\Chow{X}$ \cite{Bar75}. 
A curve $C$ is a one-dimensional projective complex space that is irreducible and reduced.
If $X$ is a complex space, we will denote an effective $1$-cycle by $\sum \alpha_l C_l$ where the $C_l \subset X$ are curves
and the coefficients $\alpha_l$ are positive integers, unless explicitly mentioned otherwise. 

\begin{definition}
Let $X$ be an irreducible and reduced complex space. A differential form $\omega$ of type $(p,q)$
is a differential form (of class $\sC^{\infty}$) of type $(p,q)$ on the smooth locus $X_{\nons}$ such that for every point
$x \in X_{\sing}$ there exists an open neighbourhood $x \in U \subset X$ and a closed embedding $i_U: U \subset V$ into an open set $V \subset \mathbb C^N$ such that there exists a 
differential form $\omega_V$  of type $(p,q)$ and class $\sC^{\infty}$ on $V$ 
such that $\omega_V|_{U \cap X_{\nons}}  = \omega|_{U \cap X_{\nons}}$.
\end{definition}

We denote the sheaf of $(p,q)$-forms by $\sA^{p,q}_X$, the cosheaf of forms with compact support is denoted by $\sA_c^{p,q}$. 
We define analogously forms of degree $r$ and denote the corresponding sheaf by $\sA^r_X$. We denote by $\sA_{\R, X}^r$ the sheaf
of real forms of degree $r$. For details, more information and references we refer to \cite[p.181]{AG05} and \cite{Dem85}.

Let $X$ be an irreducible and  reduced complex space of dimension $n$. Then the notion of a current of degree $s$ (resp. bidegree $(p,p)$ or bidimension $(n-p,n-p)$) is well-defined, see
\cite{Dem85}, as well as the notion of closedness and positivity. 
We will denote by $\sD^{p,q}(X)$ the space of currents of bidegree $(p,q)$.

Let now $X$ be a normal complex space and $\omega$ a form of type $(p,q)$ on $X$. Let
$\holom{\pi}{\hat X}{X}$ be a morphism from a complex manifold $\hat X$. 
The pull-back $\pi^* \omega$ is then defined as follows:
every point $x \in X$ 
has an open neighborhood $U$ admitting a closed embedding $i_U: U \subset V$ into an open set $V \subset \mathbb C^N$
and there exists a form $\omega_V$ on $V$ such that $\omega|_{U \cap X_{\nons}} = \omega_V|_{U \cap X_{\nons}}$.  
We can thus define the pull-back $\pi^* \omega$ on $\fibre{\pi}{U}$ by setting
$$
(\pi^* \omega)|_{\fibre{\pi}{U}} := (i_U \circ \pi|_{\fibre{\pi}{U}})^* \omega_V. 
$$
By \cite[Lem.1.3.]{Dem85} these local definitions glue together, so we obtain a pull-back map:
$$
\pi^*: \sA^{p,q}(X) \rightarrow \sA^{p,q}(\hat X).
$$
In particular if $\pi$ is proper, we can define the direct image map
$$
\pi_* : \sD^{p,q} (\hat X) \rightarrow \sD^{p,q}(X), \ T \mapsto \pi_*(T) 
$$
by setting
$$ 
\pi_*(T)(\omega) = T(\pi^*(\omega))  \qquad \forall \ \omega \in \sA_c^{n-p,n-q}(X).
$$ 

It is well known that the theory of Lelong numbers and Lelong level sets, as presented in \cite{Siu74} and \cite{Dem87}  works perfectly on any reduced complex space.
Given a reduced complex space $X$ and a positive $(1,1)-$current $T$ on $X$, choose local open sets $U$ with embeddings $i: U \to \bC^N.$ Then define the Lelong level set
$E_c(T) $ locally to be $E_c(i_*(T))$. Since Lelong numbers are invariant under biholomorphic maps, see e.g. \cite[2.2]{Dem87}, the local definitions are independent of all choices
and define globally analytic sets $E_c(T)$. 

The notion of a singular K\"ahler space was first introduced by Grauert \cite{Gra62}. 

\begin{definition} \label{definitionkaehler}
An irreducible and reduced complex space $X$ is K\"ahler if there exists a K\"ahler form $\omega$, i.e. a positive closed real $(1,1)$-form $\omega \in \mathcal A_\R^{1,1}(X)$ such that the following holds: for every point
$x \in X_{\sing}$ there exists an open neighbourhood $x \in U \subset X$ and a closed embedding $i_U: U \subset V$ into an open set $V \subset \mathbb C^N$, and  a strictly plurisubharmonic $C^{\infty}$-function $f : V \rightarrow \C$ with 
$ \omega|_{U \cap X_{\nons}} = (i \partial \overline \partial f)|_{U \cap X_{\nons}}$. 
\end{definition} 

\begin{remark} \label{remarkkaehler}\cite[1.3.1]{Var89}
If $X$ is a compact K\"ahler space and $\pi: \hat X \to X$ a projective morphism, then $\hat X$ is again K\"ahler. In particular, $X$ has a K\"ahler desingularisation.
A subvariety of a K\"ahler space is also K\"ahler.
\end{remark}

\section{The dual K\"ahler cone}

\subsection{Bott-Chern cohomology}

In this section we review the cohomology groups that replace the N\'eron-Severi space $\NS(X) \otimes \R$
from the projective setting. 

\begin{definition}  \cite[Defn. 4.6.2]{BG13}
Let $X$ be an irreducible reduced complex space. Let $\sH_X$ be the sheaf of real parts of holomorphic
functions multiplied\footnote{We ``twist'' the definition from \cite{BG13} in order to get a group that injects in
$H^2(X, \R)$ rather than $H^2(X, i \R)$} with $i$. A $(1,1)$-form (resp. $(1,1)$-current) with local potentials on $X$ is a global section
of the quotient sheaf $\sA^0_X/\sH_X$ (resp. $\mathcal D_X/\sH_X$). We define the Bott-Chern cohomology
$$
H^{1,1}_{\rm BC}(X) := H^1(X, \sH_X).
$$
\end{definition}

\begin{remark} \label{bott} {\rm 
Using the exact sequence 
$$
0 \rightarrow \sH_X \rightarrow \sA^0_X \rightarrow \sA^0_X/\sH_X \rightarrow 0,
$$
and the fact that $\sA^0_X$ is acyclic, we obtain a surjective map
$$
H^0(X,  \sA^0_X/\sH_X) \rightarrow H^1(X, \sH_X).
$$
Thus we can see an element of the Bott-Chern cohomology group as a closed $(1,1)$-form with local potentials
modulo all the forms that are globally of the form $dd^c u$.  Using the exact sequence
$$
0 \rightarrow \sH_X \rightarrow \mathcal D_X \rightarrow \mathcal D_X/\sH_X \rightarrow 0,
$$
we see that one obtains the same Bott-Chern group, if we consider $(1,1)$-currents $T$ with local potentials, which is to say that 
locally $T = dd^cu$ with $u$ a distribution. }
\end{remark}

{\bf Notation.} If $\alpha$ is a smooth form with local potentials or $T$ a current with local potentials, then its
class in $H^{1,1}_{\rm BC}(X)$ will be denoted by $[\alpha]$, resp. $[T].$ 

A fundamental property (shown in \cite[Lemma 4.6.1]{BG13}, \cite[Prop.2.1]{Fuj96}) is that
for any proper bimeromorphic morphism \holom{\pi}{X}{Y} between normal complex spaces we have
\begin{equation} \label{piRO}
\pi_* (\sH_X) = \sH_Y.
\end{equation}
Suppose now that $X$ is a compact normal space with at most rational singularities in the Fujiki class $\mathcal C$. 
Consider the exact sequence
\begin{equation} \label{standardRO}
0 \rightarrow \R \rightarrow \sO_X \rightarrow \sH_X \rightarrow 0,
\end{equation}
then the long exact sequence in cohomology yields
$$
0 \rightarrow H^1(X, \R) \rightarrow H^1(X, \sO_X) \rightarrow H^1(X, \sH_X)
\rightarrow H^2(X, \R) \rightarrow \ldots
$$
Since $H^1(X, \R) \rightarrow H^1(X, \sO_X)$ is an 
isomorphism\footnote{Since $X$ has rational singularities we have $H^1(\hat X, \R) \simeq H^1(X, \R)$ and 
$H^1(\hat X, \sO_{\hat X}) \simeq H^1(X, \sO_X)$ where $\hat X \rightarrow X$ is a resolution by a 
compact K\"ahler manifold. Now we apply the classical result from the K\"ahler case.}, 
we obtain an injection
\begin{equation} \label{BCreal}
H^1(X, \sH_X)
\hookrightarrow H^2(X, \R).
\end{equation}
Let now \holom{\varphi}{X}{Y} be a proper bimeromorphic morphism between
compact normal spaces in class $\mathcal C$ with at most rational singularities. 
Using \eqref{piRO}, the Leray spectral sequence yields an exact sequence
$$
0 \rightarrow H^1(Y, \sH_Y) \rightarrow H^1(X, \sH_X) \rightarrow H^0(Y, R^1 \varphi_*(\sH_X)).
$$
In particular we have an injection between the Bott-Chern spaces
\begin{equation} \label{BCinjective}
H^{1,1}_{\rm BC}(Y) \to H^{1,1}_{\rm BC}(X).
\end{equation}
Since $Y$ has only rational singularities and $\varphi$ is bimeromorphic, the Grothendieck spectral sequence immediately 
yields
$$ R^q\varphi_*(\sO_X) = 0$$
for $q \geq 1$. Consequently 
$R^1 \varphi_* (\R)=0$, so the Leray spectral sequence yields an exact sequence
$$
0 \rightarrow H^2(Y, \R) \rightarrow H^2(X, \R) \rightarrow H^0(Y, R^2 \varphi_*(\R)).
$$
Moreover the push-forward of the exact sequence
\eqref{standardRO} yields an isomorphism
$$
R^1 \varphi_*(\sH_X) \simeq R^2 \varphi_*(\R).
$$
Combining this isomorphism and the sequences above with the injection \eqref{BCreal} we obtain a commutative diagram
\begin{equation}
\label{commutative}
\xymatrix{
0 \ar[r] & H^1(Y, \sH_Y) \ar[r] \ar @{^{(}->}[d] & H^1(X, \sH_X) \ar[r]  \ar @{^{(}->}[d] & H^0(Y, R^1 \varphi_*(\sH_X)) \ar[d]^{\simeq}
\\
0 \ar[r] & H^2(Y, \R) \ar[r] & H^2(X, \R) \ar[r] & H^0(Y, R^2 \varphi_*(\R))
}
\end{equation}
We will now use the claim $(*)$ from the proof of \cite[Thm.12.1.3]{KM92}: given an element $S \in H^2(X, \R)$ such that
its image in $H^0(Y, R^2 \varphi_*(\R))$ is non-zero, there exists an element $C \in H_2(X/Y, \R)$ such that
$S \cdot C \neq 0$. Since $H_2(X/Y, \R)$ is generated by classes of algebraic cycles \cite[Thm.12.1.3]{KM92} we obtain
the following.

\begin{lemma} \label{pullback} 
Let  \holom{\varphi}{X}{Y} be a proper bimeromorphic morphism between
compact normal spaces in class $\mathcal C$ with at most rational singularities. 
Then we have an injection
$$
\varphi^*: H^{1,1}_{\rm BC}(Y) = H^1(Y, \sH_Y) \hookrightarrow H^1(X, \sH_X) = H^{1,1}_{\rm BC}(X) 
$$
and
$$
\mbox{\rm Im} \  \varphi^* = \{
\alpha \in H^1(X, \sH_X) \ | \ \alpha \cdot C = 0 \quad \forall \ C \subset X \ \mbox{curve s.t.} \ \varphi(C)=pt
\}. $$
Furthermore, let $\alpha \in H^1(X, \sH_X) \subset H^2(X, \R)$ be a class such that $\alpha = \varphi^* \beta$ with
$\beta \in H^2(Y, \R)$. Then there exists a smooth real closed $(1,1)$-form $\omega_Y$ on $Y$
such that $\alpha = \varphi^* [\omega_Y ]$.
\end{lemma}

\begin{proof} A simple diagram chase in \eqref{commutative}. 
\end{proof}

By \cite[Prop.4.6.3]{BG13} we can push-forward currents with local potentials:

\begin{lemma} \label{lemmapush} 
Let  \holom{\varphi}{X}{Y} be a proper bimeromorphic morphism between
compact normal spaces in class $\mathcal C$ with at most rational singularities. 
Let $\alpha \in H^1(Y, \sH_Y)$ be a class, and let $\omega$ be a positive closed (1,1)-current with local potentials
on $X$ such that
$$
[\omega] = \pi^* \alpha. 
$$
Denote by $E \subset Y$ the image of the $\varphi$-exceptional locus. Then $\varphi_*(\omega)|_{Y \setminus E}$
has a unique extension  $\widehat{\varphi_*(\omega)}$ as a closed positive current with local potentials to $Y$ and
$[\widehat{\varphi_*(\omega)}]=\alpha$.
\end{lemma}

\begin{remark} \label{remarksmooth}
Using the existence of local potentials we can apply regularization techniques known from the smooth case \cite{Dem92}.
For example consider $Y$ a normal complex space with rational isolated singularities, and let \holom{\pi}{X}{Y}
be a desingularisation. Suppose that there exists a class $\alpha \in H^1(Y, \sH_Y)$ and a K\"ahler form $\omega$ on $X$
such that $[\widehat{\pi_*(\omega)}]=\alpha$. Then $\alpha$ can be represented by a smooth K\"ahler form: if $y \in Y_{sing}$ is a point, choose
a local embedding of a small neighborhood in $\C^N$, and let $\varphi$ be the local potential of $\widehat{\pi_*(\omega)}$.
By construction $\varphi$ is smooth and strictly psh in the complement of $y$, so we can take a regularized maximum with
some smooth strictly psh function $\varphi_s$ to obtain a local potential which is smooth, strictly psh and glues to a global form.
\end{remark}

To make the connection to the projective case, we 
define

\begin{definition} \label{defN1} 
Let $X$ be a normal compact complex space. Then we set
$$ 
N^1(X) := H^{1,1}_{\rm BC}(X).  
$$ 
A class $\alpha \in N^1(X)$ is pseudoeffective if it can be represented by a current $T$ which is locally
of the form $\partial \pp \varphi$ with $\varphi$ a psh function.
\end{definition}

\begin{remark} 
If $X$ has only rational singularities, we have seen that 
$$ N^1(X) \subset H^2(X,\bR).$$
Using this inclusion one can define, as in the smooth case,
the intersection product on $N^1(X)$ via the the cup-product for $H^2(X, \R)$. 
\end{remark} 

Note that for $X$ projective, the space $N^1(X)$
is usually defined differently, namely as the $\R$-vector space generated by the classes of irreducible divisors. 
Since we are dealing with general compact complex spaces, this space is often too small to be useful.

\subsection{The dual K\"ahler cone} 

Let $X$ be a normal compact complex space, and let $C \subset X$ be an irreducible curve.
We associate the current of integration $T_C$ by setting
$$
T_C(\omega) = \int_C \omega
$$
for $\omega$ any real closed $(1,1)$-form on $X$. If $X$ is smooth it is well-known that $T_C$
is closed, one sees easily that this also holds in the singular case:
let $\holom{\pi}{\hat X}{X}$ be a desingularisation. If $C$ is not contained in the singular locus of $X$,
let $\hat C$ be its strict transform. Then $T_{\hat C}$ is a closed current, so
$\pi_* (T_{\hat C}) =T_C$ is closed. If $C \subset {\rm Sing}(X)$, then take $\hat C$ to be any curve
such that $\pi(\hat C) = C$ (which exists since the resolution $\pi$ is a projective morphism).
Then we have $\pi_*(T_{\hat C}) = d T_C$, where $d$ is the degree of $\hat C \rightarrow C$. 
Hence $T_C$ is closed.

\begin{definition} \label{defNA} 
Let $X$ be a normal compact complex space in class $\mathcal C.$ We define $N_1(X) $ to be the vector space of real closed currents of bidimension $(1,1)$ modulo the
 following equivalence relation: 
 $ T_1 \equiv  T_2 $ if and only if
 $$ T_1(\eta) = T_2(\eta)$$
 for all real closed $(1,1)$-forms $\eta$ with local potentials.  \\
We define $\NAX \subset  N_1(X)$ to be the closed cone generated by the classes of positive closed currents. 
We define  the Mori cone 
$$
\NEX \subset \NAX 
$$
as the closure of the cone generated by the currents of integration $T_C$ where $C \subset X$ is an irreducible curve.
\end{definition}

\begin{remark*}
The inclusion $\overline {NE}(X) \subset \overline {NA}(X)$   
can be strict even if $X$ is a projective manifold: in this case $\overline {NE}(X)$
generates a vector space of dimension equal to the Picard number $\rho(X)$, while $\NAX$ generates
$H^{n-1,n-1}(X)$.
\end{remark*}

Let $\omega$ be a real closed $(1,1)$-form on $X$ with local potentials. Then we can define 
$$
\lambda_\omega \in N_1(X)^*, \ [T] \mapsto T(\omega).
$$
If $T(\omega)=0$ for all  closed currents $T$ of bidimension $(1,1)$, we have $\lambda_\omega=0$. Thus we have a well-defined canonical map
$$
\Phi: N^1(X) \to N_1(X)^*, \ [\omega] \mapsto \lambda_\omega.
$$

\begin{proposition} \label{dual}
Let $X$ be a normal compact complex space with rational singularities of dimension~$n$ in class $\mathcal C$. 
Then the canonical map $\Phi$ is an isomorphism. In particular $N_1(X) $ is finite-dimensional.
Moreover, given a linear map $\lambda: N_1(X) \to \mathbb R$, 
there exists a real closed $(1,1)$-form $\omega$ such that 
$$
\lambda([T]) = T(\omega).
$$

If $\mu: X' \to X$ is a bimeromorphic map from a normal compact complex space $X'$ in class $\mathcal C$, then the natural linear map
$$ 
\mu_*: N_1(X') \to N_1(X)
$$ 
is surjective.
\end{proposition} 

\begin{proof} Obviously, the vector space $N^1(X) $ is finite-dimensional. 

Let $\omega$ be a real closed $(1,1)$-form with local potentials such that $\Phi([\omega]) = 0$. 
Then $T(\omega) = 0$ for all closed currents $T$ of bidegree $(n-1, n-1)$.  
By definition of $N^1(X)$ this gives $[\omega] = 0$. Thus $\Phi$ is injective. 
 
Let us now show that $\Phi$ is surjective. 
Since $N^1(X)$ is finite-dimensional, it suffices to show that $\dim N_1(X) \leq  \dim N^1(X).$ However this is clear, since we can prove as above that the natural map 
$$ 
\psi: N_1(X) \to N^1(X)^*, \ [T] \mapsto \psi([T])([\omega]) = T(\omega)
$$  
is well-defined and injective.

For the proof of the second statement, let $\holom{\pi}{\hat X}{X}$ be a desingularisation of $X$ that factors through $\mu$.
By Lemma \ref{pullback},  the pull-back 
$$
N^1(X) \to N^1(\hat X)
$$
is injective, so by duality the push-forward 
$\pi_*: N_1(\hat X) \to N_1(X)$ 
is surjective. By functoriality of the push-forward this shows that 
$$ 
\mu_*: N_1(X') \to N_1(X)
$$ 
is also surjective.
\end{proof} 

For the convenience of the reader we recall the definition of nefness, cf.\cite[Def.3]{Pau98}.

\begin{definition} \label{defnef}  
Let $X$ be an irreducible reduced compact complex space, and let $u \in H^{1,1}_{\rm BC}(X)$ be 
a class represented by a form $\alpha$ with local potentials. Then $u$ is nef if for some positive $(1,1)-$form $\omega $ on $X$ and
for every $\epsilon > 0$ there exists $f_{\epsilon} \in \sA^0(X)$ such that
$$ \alpha + i \partial \overline{\partial} f_{\epsilon} \geq - \epsilon \omega.$$
\end{definition}

\begin{definition} \label{definitionnefcone}
Let $X$ be a normal compact complex space in class $\mathcal C$. 
We denote by ${\rm Nef}(X) \subset N^1(X)$ the cone generated by cohomology classes which are
nef. 
\end{definition}

\begin{remark} \label{remarkkaehlerconenef}
If $X$ is a normal compact K\"ahler space we can also consider the open cone $\sK$ generated by
the classes of K\"ahler forms. In this case we know\footnote{The statement in \cite[Prop.6.1.iii)]{Dem92} is for compact manifolds,
but the proof works in the singular setting.} that
$$
{\rm Nef}(X)  = \overline {\sK}.
$$
\end{remark}

\begin{lemma} \label{lemmanef}
Let $X$ be a normal compact threefold in class $\mathcal C$, and let
$\mu: X' \to X$ be a bimeromorphic holomorphic map from a normal compact threefold.
Let $\eta$ be a real closed $(1,1)$-form with local potentials on $X$.
Then the class $[ \eta ] \in N^1(X)$ is nef if $\mu^* [ \eta ] \in N^1(X')$
is nef.
\end{lemma}

\begin{proof}
Since $\mu^* [ \eta ]$ is nef, the class $[ \eta ] = \mu_* \mu^* [ \eta ]$ is pseudoeffective.
By \cite[Prop.3.3(iv)]{DP04} the class $[ \eta ]$ is nef if the restriction to every Lelong set is nef.
Note now that any irreducible component $Z$ of a Lelong set is the image of a Lelong set of $\mu^* [ \eta ]$
or contained in the image of the $\mu$-exceptional locus. In the first case the restriction is nef,
since $\mu^* [ \eta ]$ is nef. In the second case, since $X$ is a threefold, we see that $Z$ is a point or a curve.
The restriction to a point is obviously nef, so suppose that $Z$ is a curve and let $Z' \subset X'$ be an irreducible
curve such that the induced map $Z' \rightarrow Z$ is finite of degree $d$. Then we have
$$
\int_Z \eta|_Z = \frac{1}{d} \int_{Z'} \pi^* \eta|_{Z'} \geq 0
$$ 
since $\pi^* [ \eta ]$ is nef. Thus $[ \eta ]|_Z$ has non-negative degree on the curve $Z$, so it is nef.
\end{proof}

\begin{proposition} \label{ext} 
Let $X$ be a normal compact threefold with rational singularities in class $\mathcal C$. 
Let $\mu: X' \to X$ be a bimeromorphic holomorphic map from a normal compact threefold.
Then we have
$$
\mu_*(\overline {NA}(X')) = \overline {NA}(X).
$$ 
\end{proposition} 

\begin{proof} 
Let $\holom{\pi}{\hat X}{X}$ be a desingularisation of $X$ that factors through $\mu$ and such that $\hat X$ is a compact
K\"ahler manifold.
By the functoriality of the push-forward it is sufficient to prove that
$$
\pi_*(\overline {NA}(\hat X)) = \overline {NA}(X).
$$ 
We clearly have $\pi_*(\overline {NA}(\hat X)) \subset \overline {NA}(X)$.
Arguing by contradiction we assume that $\pi_*(\overline {NA}(\hat X)) $ is a proper subcone of $\overline {NA}(X).$ 
The space $N_1(X)$ being finite-dimensional, there exists a linear map
$$ \lambda: N_1(X) \to \mathbb R$$
which is non-negative on  $\pi_*(\overline {NA}(\hat X)),$ and $\lambda(T_0) < 0$ for some $T_0 \in  \overline {NA}(X).$ 
By Proposition \ref{dual}, there exists a real closed $(1,1)$-form $\eta \in \sA^{1,1}(X)$ with local potentials such that $\lambda([T]) = T(\eta) $ for all $[T] \in N_1(X).$ 
We define
$$
\hat \lambda: N_1(\hat X) \to \mathbb R
$$ by 
$\hat \lambda([\hat T]) = \hat T(\pi^* \eta).$ Then we have 
$$ \hat \lambda([\hat T]) = \hat T(\pi^*(\eta)) = \pi_*(\hat T)(\eta) \geq 0$$
for all positive closed currents $\hat T$ on $\hat X$, so  
$\hat \lambda$ is non-negative on $\overline {NA}(\hat X)$.  
Since $\overline {NA}(\hat X)$ is dual to the K\"ahler cone  (a well-known consequence of \cite[Cor.0.3]{DP04}; see e.g. \cite[Prop.1.8]{OP04}),
we conclude that $[ \pi^*(\eta) ] $ is a nef class. 
Thus $[ \eta ]$ is nef by Lemma \ref{lemmanef}, hence
$[\eta]$ is non-negative on
$\overline {NA}(X)$, a contradiction.
\end{proof}

\begin{proposition} \label{dual1}
Let $X$ be a normal compact threefold with rational singularities in class $\mathcal C$. 
Then the cones ${\rm Nef}(X)$ and $ \overline{NA}(X)$ are dual via the canonical 
isomorphism $\Phi: N^1(X) \to N_1(X)^*$, constructed in Proposition \ref{dual}. 
\end{proposition} 

\begin{proof} 
For a closed convex cone $V$ in some finite-dimensional real vector space we have $V=V^{**}$ \cite[Lemma 6.7]{Deb01}, so it is
sufficient to prove that
$$
{\rm Nef}(X) = \overline{NA}(X)^*.
$$
It is clear that ${\rm Nef}(X) \subset \overline{NA}(X)^*$, so we are left to prove the other inclusion.
We consider a linear form $\lambda: N_1(X) \to \bR$ such that $\lambda(T) \geq 0$
for all $[T] \in \NAX$. Choose a closed $(1,1)-$form $\omega$ with local potentials such that $T(\omega) = \lambda([T])$ then
we want to show that $[\omega] \in {\rm Nef}(X).$ 

Let $\pi: \hat X  \to X$ be a desingularisation such that $\hat X$ is a compact K\"ahler manifold, 
then $\pi^* \omega$ defines a linear form on $N_1(\hat X)$ which is non-negative
on $\overline{NA}(\hat X)$. Thus $[\pi^* \omega]$ is a nef class by \cite[Cor.0.3]{DP04},
so the class $[\omega]$ is nef by Lemma \ref{lemmanef}.
\end{proof} 

In the K\"ahler case, the cone $\NA{X}$ is usually defined right away as the dual of the K\"ahler cone. 

\begin{corollary} \label{corollarykaehlerform} 
Let $X$ be a normal compact K\"ahler threefold with rational singularities . 
Suppose that $[\eta] \in N^1(X)$ is strictly positive on $\overline {NA}(X) \setminus 0$.
Then $[\eta]$ is a K\"ahler class, i.e. can be represented by a K\"ahler form $\omega$.   
\end{corollary} 

\begin{proof}
By Proposition \ref{dual1} we see that $[\eta]$ lies in the interior of the cone ${\rm Nef}(X)$.
Since $X$ is K\"ahler, we have ${\rm Nef}(X)  = \overline {\sK}$ (cf. Remark \ref{remarkkaehlerconenef}), so $[\eta]$ is a K\"ahler class. 
\end{proof}

\begin{theorem} \label{kaehler1} Let $X$ be a normal compact threefold with rational singularities in class $\sC$.
Let $\eta$ be a closed real $(1,1)$-form with local potentials such that $T(\eta) > 0$ for all positive closed currents $T \in \sD^{2,2}(X) \setminus \{0\}$. 
Then  $[ \eta ]$ contains a K\"ahler form, in particular $X$ is a K\"ahler space. 
\end{theorem} 

\begin{proof} 
Let $\pi: \hat X \to X$ be a desingularisation such that $\hat X$ is a compact K\"ahler manifold and 
such that $\pi$ is a projective morphism with exceptional locus a simple normal crossings divisor.
Up to blowing up further there exists an effective $\pi$-exceptional divisor $D$ on $\hat X$ such that $-D$ is $\pi-$ample. 
Moreover for large $N$ the class
$$ 
[N \pi^*(\eta) - D] 
$$ 
is strictly positive on $\overline {NA}(\hat X)$.
By \cite[Cor.0.4]{DP04} there exists a K\"ahler form $\hat \omega$ on $\hat X$ such that
$$
[N \pi^*(\eta) - D] = [ \hat \omega ].
$$
Thus we can write
$$ 
\hat \omega  = N \pi^*(\eta) - T_D + \partial \pp \hat f
$$
where $\hat f \in \sA^0(\hat X)$ and $T_D$ the current of integration associated with $D$. 
Applying $\pi_*$ it follows that
$$ 
\pi_* (\hat \omega)  = N \eta + \partial \pp R
$$
with $R = \pi_*(\hat f)$. 
Let $E \subset X$ be the image of the $\pi$-exceptional locus. 
By Lemma \ref{lemmapush} the current $\pi_*(\omega) \vert_{X \setminus E}$ has an extension with 
local potentials and $[\pi_*(\hat \omega)] = [N \eta ]$. 
By \cite[Prop.3.3(iii)]{DP04} (cf. Remark \ref{remarksmooth} for regularisation arguments in the singular setting),
the class $[ \pi_* (\hat \omega) ] = [N \eta ]$ of the K\"ahler current $\pi_*(\hat \omega)$
contains a K\"ahler form  if and only if $[ \pi_* (\hat \omega) ]|_Z 
$ is a K\"ahler class for every irreducible
component $Z$ of the Lelong level sets. 
The K\"ahler current $\pi_*(\hat \omega) $ is smooth outside $E$, hence any Lelong level set $Z$ is contained in $E$. Since $X$ is a threefold we see that $Z$ is either a point or a curve.
The case $Z$ being a point is trivial, so suppose that $Z$ is a curve.
Since the current of integration $T_Z$ is a non-zero positive closed current
we have $T_Z(\eta) > 0$ by assumption. Thus the restriction $i_Z^*(\eta)$ to the curve $Z$ has positive degree, i.e.
$i_Z^*(\eta)$ is a K\"ahler class on the curve $Z$.
\end{proof} 

\begin{theorem} \label{kaehler2} Let $X$ be a normal compact threefold in class $\sC$ with only isolated rational singularities.
Let $\eta \in \sA^{1,1}(X)$ be a closed real (1,1)-form with local potentials such that
$T(\eta) > 0$ for all $[T] \in \overline {NA}(X) \setminus 0.$ 
Suppose that for every irreducible curve $C \subset X$ we have $[C] \neq 0$ in $N_1(X)$.
Then $[ \eta ]$ is represented by a K\"ahler class, in particular $X$ is K\"ahler. 
\end{theorem} 

\begin{proof} 
{\em Step 1. Suppose that $X$ is smooth.} We argue by contradiction. 
By Theorem~\ref{kaehler1} there exists a positive closed current $T \neq 0$ such that $T(\eta) = 0$.
By our assumption this implies that $[T] = 0$ in $N_1(X)$. \\
Let \holom{\pi}{\hat X}{X} be a modification of $X$ such that $\hat X$ is a compact K\"ahler manifold.
Let $S \subset  X$ be the image of the $\pi$-exceptional locus. 
We write $S = \bigcup C_j \cup A$ with $C_j$ the irreducible components of dimension 1 and
$A$ a finite set. We consider the positive closed currents $\chi_S T$ and $\chi_{X \setminus S}T$. 
Since $\eta$ is non-negative on $\NAX$ the
decomposition
$$  T = \chi_S T + \chi_{X \setminus S}T $$
implies that $\chi_S T(\eta) = \chi_{X \setminus S}T(\eta) = 0$. By our assumption this implies 
$[\chi_ST] = [\chi_{X \setminus S}T] = 0. $
By a theorem of Siu \cite{Siu74} we have 
$$ \chi_S T = \sum \alpha_j T_{C_j} $$
with $\alpha_j \geq 0 $ and $T_{C_j}$ denoting the current of integration over $C_j$. 
Since $[\chi_S T] = 0$ and every curve class is non-zero, we conclude that $\alpha_j = 0$ for all $j$. Thus we have $\chi_S T = 0$. 
\vskip .2cm 
It remains to prove that $\chi_{X \setminus S}T = 0.$ 
Fix a K\"ahler form $\hat \omega $ on $\hat X$ and consider the positive closed current
$$R = \pi_*(\hat \omega).$$
Choose a sequence $(\lambda_k)$ of positive $(2,2)-$forms converging weakly to the positive current $\chi_{X \setminus S}T,$  such that 
$d \lambda_k $ converges weakly to $d \chi_{X \setminus S}T = 0,$ see e.g. \cite[p.373]{GH78}. 
Using Hodge decomposition of forms, we may choose a sequence $(\mu_k)$ of $3-$forms such that $\lambda_k - d \mu_k$ converges weakly to $0.$ 
In fact, consider the Hodge decomposition for forms, see e.g. \cite[p.84]{GH78} (in the analytic $(p,q)-$setting) 
$$ \lambda_k = \sH(\lambda_k) \oplus dd^*G(\lambda_k) \oplus d^*dG(\lambda_k),$$
where $\sH$ is the harmonic projection and $G$ is the Green operator. 
Since $[\chi_{X \setminus S}T] = 0,$ the sequence $\sH(\lambda_k)$ converges to $0$; and since $(d\lambda_k )$ converges weakly to $0$, the
sequence $(d^*dG(\lambda_k)) = (G(d^*d\lambda_k))$ converges to $0$, too. Thus we can set $\mu_k := d^*G(\lambda_k).$ 
\\
Choose furthermore a sequence $\varphi_j$ of non-negative $\mathcal C^{\infty}$ functions with compact support in $X \setminus S$ converging increasingly to 
$\chi_{X \setminus S}.$ Then $\varphi_j R$ is a positive (non-closed) smooth $(1,1)-$form on $X$, hence $\chi_{X \setminus S}T(\varphi_jR) \geq 0.$
Moreover, the sequence $(\chi_{X \setminus S}T(\varphi_jR))$ is increasing. 
We are going to show that $$\chi_{X \setminus S}T(\varphi_jR) = 0$$ for all $j.$ Then the support of $(\chi_{X \setminus S}T)$ 
is contained in $S$, hence $$\chi_{X \setminus S}T = 0.$$
So assume that $\chi_{X \setminus S}T(\varphi_jR) \ne 0$ for some, hence almost all, $j.$ Thus we find positive constant $C$ and a number $j_0$ such that 
$$\chi_{X \setminus S}T(\varphi_jR)  \geq C$$
for $j \geq j_0.$ Thus for all $j \geq j_0$ there is a number $k(j)$ such that 
$$ 
R(\varphi_j \lambda_{k(j)}) = \int_X \lambda_{k(j)} \wedge \varphi_j R \geq \frac{C}{2}.
$$
On the other hand, 
$$ R((1-\varphi_j) \lambda_{k(j)}) = \int_{\hat X} \pi^*((1-\varphi_j) \lambda_{k(j)}) \wedge \hat \omega $$
converges to $0$ for $j \rightarrow \infty$ and 
$$ 
\lim_{j \rightarrow \infty} R(\lambda_{k(j)}) = \lim_{j \rightarrow \infty} \int_{\hat X} \pi^*(\lambda_{k(j)}) \wedge \hat \omega = \lim_{j \rightarrow \infty} \int_{\hat X} \pi^*(\lambda_{k(j)} - d\mu_{k(j)}) \wedge \hat \omega  = 0,$$
giving a contradiction.

{\em Step 2. Reduction to the smooth case.} Let $\holom{\mu}{\hat X}{X}$ be a desingularisation, and let
$E$ be an effective $\mu$-exceptional divisor such that $-E$ is $\mu$-ample. Let $\hat C \subset \hat X$
be an irreducible curve. If $\hat C$ is $\mu$-exceptional, then $-E \cdot \hat C>0$, so $[\hat C] \neq 0$ in $N_1(\hat X)$.
If $\hat C$ is not $\mu$-exceptional, then we have
$$
\mu_* ([\hat C]) = [\mu(\hat C)] \neq 0, 
$$
so $[\hat C] \neq 0$ in $N_1(\hat X)$. We claim that there exists a number $n \in \N$ such that 
$$T(\mu^* \eta - \frac{1}{n} E)>0$$
for all $[T] \in \overline {NA}(\hat X) \setminus 0$. Assuming this for the time being, let us see how to conclude: 
by Step 1 we know that $\hat X$ is a K\"ahler manifold, actually $[\mu^* \eta - \frac{1}{n} E) ]$ is a K\"ahler class. Thus there exists a K\"ahler form 
$$\hat \omega \in [\mu^* \eta - \frac{1}{n} E) ].$$
The push-forward $\mu_*(\hat \omega)$ is a K\"ahler current on $X$ and the Lelong level sets of  $\mu_*(\hat \omega)$
are contained the singular locus of $X$. Since $X$ has isolated singularities, we see that $[ \mu_*(\hat \omega) ]|_Z$
is a K\"ahler class for every $Z$ an irreducible component of the Lelong level sets. By \cite[Prop.3.3(iii)]{DP04} (cf. Remark \ref{remarksmooth}) this shows
that $[ \mu_*(\hat \omega) ]$ contains a K\"ahler form.

{\em Proof of the claim.} Fixing a hermitian metric $h$ on $\hat X$, we see that it is sufficient
to prove that $$T(\mu^* \eta - \frac{1}{n} \psi)>0$$
for all $[T] \in (\overline {NA}(\hat X) \cap H)$. Here $\psi$ is a smooth representative of $E$ in $H^{1,1}_{\rm BC}(\hat X)$ and $H$ is the hypersurface of classes having norm one 
with respect to a fixed norm on the finite-dimensional space $N^1(\hat X)$. Note that $\overline {NA}(\hat X) \cap H$ is a compact set, in
particular the function 
$$
e: \overline {NA}(\hat X) \cap H \rightarrow \R, \ T \mapsto T(E)
$$
is bounded from above. Arguing by contradiction we assume that for all $n \in \N$  there exists a class 
$[T_n] \in (\overline {NA}(\hat X) \cap H)$, represented by a positive closed current $T_n,$  such that
\begin{equation} \label{curr}
T_n(\mu^* \eta - \frac{1}{n} E) \leq 0.
\end{equation}
Since $\overline {NA}(\hat X) \cap H$ is compact, we can suppose, up to taking subsequences,
that the sequence $[T_n]$ converges to a limit $[T_\infty] \neq 0$, represented by a positive closed current $T_{\infty}$. 
Since $T_n(E)$ is bounded, the equation \eqref{curr} simplifies to
$$
T_\infty(\mu^* \eta) \leq 0.
$$
Since $\eta$ is positive on $\NAX \setminus \{0\},$ this implies $[\mu_*(T_\infty)]=0$. Since
$-E$ is $\mu$-ample we conclude that $T_\infty(E)<0$. Yet by continuity this implies that
$T_n(E)<0$ for $n \in \N$ sufficiently large. However, \eqref{curr} is equivalent to
$$
T_n(E) \geq n \ T_n(\mu^* \eta) \geq 0,
$$
a contradiction.
\end{proof}

The importance of the dual K\"ahler cone $\NAX$ in our context lies in the following 

\begin{definition} \label{definitioncontraction}
Let $X$ be a normal $\Q$-factorial compact K\"ahler space with at most terminal  singularities,
and let $\R^+ [\Gamma_i]$ be a $K_X$-negative extremal ray in $\NAX$.
A contraction of the extremal ray $\R^+ [\Gamma_i]$ is
a morphism \holom{\varphi}{X}{Y} onto a normal compact K\"ahler space
such that $-K_X$ is $\varphi$-ample and
a curve $C \subset X$ is contracted if and only if $[C] \in \R^+ [\Gamma_i]$.
\end{definition}

\section{Subvarieties of K\"ahler threefolds}

\subsection{Remarks on adjunction} \label{subsectionsurfaces}
Let $X$ be normal  $\Q$-factorial compact K\"ahler threefold with at most terminal  singularities. 
Let $S \subset X$ be a prime divisor, i.e. an irreducible and reduced compact surface. Let $m \in \N$ be the smallest positive integer such that
both $mK_X$ and $mS$ are Cartier divisors on $X$. Then the canonical divisor $K_S \in \pic(S) \otimes \Q$ is defined by 
$$
K_S := \frac{1}{m} (mK_X+mS)|_S.
$$
Since $X$ is smooth in codimension two, there
exist at most finitely many points $\{p_1, \ldots p_q \}$ where $K_X$ and $S$ are not Cartier. 
Thus by the adjunction formula $K_S$ is isomorphic to the dualising sheaf $\omega_S$ on $S \setminus  \{p_1, \ldots p_q \}$.

Let now \holom{\nu}{\tilde S}{S} be the normalisation. Then we have
\begin{equation} \label{equationconductor}
K_{\tilde S} \sim_\Q \nu^* K_S - N,
\end{equation}
where $N$ is an effective Weil divisor defined by the conductor ideal. Indeed this formula holds 
by \cite{Rei94} for the dualising sheaves. Since $\sO_{\tilde S}(\nu^* K_S)$ is isomorphic to $\nu^* \omega_S$ on the complement
of $\fibre{\nu}{p_1, \ldots p_q}$, the formula holds for the canonical divisors.

Let \holom{\mu}{\hat S}{\tilde S} be the minimal resolution of the normal surface $S$, then we have
$$
K_{\hat S} \sim_\Q \mu^* K_{\tilde S} - N',
$$
where $N'$ is an effective $\mu$-exceptional divisor \cite[4.1]{Sak81}. Thus if \holom{\pi}{\hat S}{S} is the composition $\nu \circ \mu$, there exists an
effective, canonically defined divisor $E \subset \hat S$ such that
\begin{equation} \label{equationresolve}
K_{\hat S} \sim_\Q \pi^* K_S - E.
\end{equation}
Let $C \subset S$ be a curve such that $C \not\subset S_{\sing}$. Then the morphism $\pi$ is an isomorphism in the generic point of $C$,
and we can define the strict transform $\hat C \subset \hat S$ as the closure of $C \setminus S_{\sing}$. Since $\hat C$ is an (irreducible) curve that is not contained in 
the divisor $N$ defined by the conductor, we have $\hat C \not\subset E$. By the projection formula
and \eqref{equationresolve} we obtain
\begin{equation} \label{inequalitycanonical}
K_{\hat S} \cdot \hat C \leq K_S \cdot C.
\end{equation}

\subsection{Divisorial Zariski decomposition for $K_X$}
Let $X$ be a  normal  $\Q$-factorial compact K\"ahler threefold with at most terminal singularities.
If $K_X$ is not pseudoeffective, we know by a theorem of Brunella \cite{Bru06} 
that $X$ is uniruled. In particular we have a covering family of rational curves such that $K_X \cdot C<0$.

From now on suppose that $K_X$ is pseudoeffective.
By  \cite[Thm.3.12]{Bou04} we have a divisorial Zariski decomposition\footnote{The statements in \cite{Bou04} are for complex compact manifolds, 
but generalise immediately to our
situation: take $\holom{\mu}{X'}{X}$ a desingularisation, and let $m \in \N$ be the Cartier index of $K_X$. Then
$\mu^* (mK_X)$ is a pseudoeffective line bundle with divisorial Zariski decomposition 
$\mu^* (mK_X)= \sum \eta_j S_j' + N(mK_X)'$. The decomposition of $K_X$ is defined by the push-forwards
$\mu_* (\frac{1}{m} \sum \eta_j S_j')$ and $\mu_* (\frac{1}{m} N(mK_X)')$. Since a prime divisor $D \subset X$
is not contained in the singular locus of $X$, the decomposition has the stated properties.}
\begin{equation} \label{Bdecomposition}
K_X = \sum_{j=1}^r \lambda_j S_j + N(K_X),
\end{equation}
where the $S_j$ are integral surfaces in $X$, the coefficients $\lambda_j \in \R^+$ and $N(K_X)$ is a 
pseudoeffective class which is nef in codimension one \cite[Prop.2.4]{Bou04}, that is for every surface $S \subset X$
the restriction $N(K_X)|_S$ is pseudoeffective.

\begin{lemma} \label{lemmasurfaces}
Let $X$ be a  normal  $\Q$-factorial compact K\"ahler threefold with at most terminal singularities
such that $K_X$ is pseudoeffective. Let $S$ be a surface 
such that $K_X|_S$ is not pseudoeffective. Then 
$S$ is one of the surfaces $S_j$ in the divisorial Zariski decomposition \eqref{Bdecomposition} of $K_X$.
Moreover $S=S_j$ is Moishezon, moreover any desingularisation $\hat S_j$ is a uniruled projective surface.
\end{lemma}

\begin{proof}
Suppose first that $S \neq S_j$ for every $j \in \{ 1, \ldots, r \}$.
Then we have
$$
K_X|_{S} = \sum_{j=1}^r \lambda_j (S \cap S_j) + N(K_X)|_{S}.
$$
Since $N(K_X)|_S$ is pseudoeffective, we see that $K_X|_{S}$ is pseudoeffective, a contradiction.
Thus we have $S=S_j$ for some $j$ and (up to renumbering) we can suppose that $S=S_1$.
We have
$$
S = S_1 = \frac{1}{\lambda_1} K_X - \frac{1}{\lambda_1} (\sum_{j=2}^r \lambda_j S_j +  N(K_X)),
$$
so by adjunction
$$
K_{S} = (K_X+S)|_{S} = \frac{\lambda_1+1}{\lambda_1} K_X|_S - \frac{1}{\lambda_1} (\sum_{j=2}^r \lambda_j (S_j \cap S) +  N(K_X)|_S).
$$
By hypothesis $K_X|_S$ is not pseudoeffective and $- \frac{1}{\lambda_1} (\sum_{j=2}^r \lambda_j (S_j \cap S) +  N(K_X)|_S)$ is anti-pseudoeffective. Thus $K_S$ is not pseudoeffective.

Let now
\holom{\pi}{\hat{S}}{S} be the composition of normalisation and minimal resolution of the normalised surface,
then by \eqref{equationresolve} there exists an effective divisor $E$ such that
$$
K_{\hat{S}} \sim_\Q \pi^* K_{S} - E. 
$$
Thus $K_{\hat{S}}$ is not pseudoeffective, in particular $\kappa(\hat S)=-\infty$. 
By Remark \ref{remarkkaehler} the surface $\hat S$ is K\"ahler, 
so it follows from the Castelnuovo-Kodaira classification that $\hat{S}$ is covered by rational curves,
in particular it is a projective surface \cite{BHPV04}. Thus $S$ is Moishezon.
\end{proof}

\begin{corollary} \label{corollaryalgebraicallynef}
Let $X$ be a  normal  $\Q$-factorial compact K\"ahler threefold with at most terminal singularities
such that $K_X$ is pseudoeffective. Then $K_X$ is nef if and only if
$$
K_X \cdot C \geq 0 
$$
for every curve $C \subset X$.
\end{corollary}

\begin{proof}
One implication is trivial. Suppose now that $K_X$ is nef on all curves $C$. 
We will argue by contradiction and suppose that $K_X$ is not nef. 
Since $K_X$ is pseudoeffective and the restriction to every curve is nef, there exists by \cite[Prop.3.4]{Bou04} an irreducible surface $S \subset X$ such that $K_X|_S$ is not pseudoeffective. 
By Lemma \ref{lemmasurfaces} a desingularisation \holom{\pi}{\hat S}{S} of the surface $S$ is projective,
and the divisor $\pi^* K_X|_{S}$ is not pseudoeffective, so there exists a covering family of curves $C_t \subset S$
such that 
$$
K_X \cdot C_t = (\pi^* K_X|_{S}) \cdot C_t < 0,
$$
a contradiction. 
\end{proof}

\subsection{Very rigid curves}

\begin{definition} \label{definitionveryrigid}
Let $X$ be a  normal $\Q$-factorial K\"ahler threefold with at most terminal  singularities.
We say that a curve $C$ is very rigid if 
$$
\dim_{mC} \Chow{X} = 0
$$
for all $m>0$.
\end{definition}

\begin{remark} \label{remarkkollar}
Koll\'ar \cite[Defn.4.1]{Kol91b} defines a curve $C \subset X$ to be very rigid if there is no flat family of one dimensional
subschemes $D_t \subset X$ parametrised by a pointed disc $0 \in \Delta$ such that 
\begin{itemize}
\item $\supp D_0=C$,
\item $\supp D_t \not\subset C$ for $t \neq 0$, and
\item $D_t$ is purely one dimensional for $t \neq 0$.
\end{itemize} 
A curve that is very rigid in the sense of Definition \ref{definitionveryrigid}, is also very rigid in the sense of Koll\'ar: using the
natural map from the normalisation of the Douady space $\mathcal H(X)$ to the cycle space $\Chow{X}$ \cite[p.121]{Bar75},
a family $(D_t)_{t \in \Delta}$ defines a positive-dimensional subset of $\Chow{X}$ through a point
corresponding to $mC$ for some $m>0$.
\end{remark}

\begin{theorem} \label{theoremveryrigid} 
Let $X$ be a normal $\mathbb Q$-factorial K\"ahler threefold with at most terminal singularities. Let $C \subset X$ be
an irreducible very rigid curve such that $K_X \cdot C < 0.$ Then $C$ is a rational curve and 
$$ K_X \cdot C \geq -1.$$ 
\end{theorem}

For the proof of this theorem we follow the arguments of Koll\'ar \cite{Kol91b}, with some
simplifications due to our special situation. 
It is actually not necessary to work in the algebraic category, all arguments work in the analytic category as well. 

\begin{proof} 
{\em Step 1. If $K_X \cdot C<-1$, then $C$ is not very rigid.}
Let $\hat X$ be the formal completion of $X$  along $C$. Choose $m$ minimal such that $mK_X$ is Cartier. 
Since $-mK_X|_C$  is ample, the restriction $-mK_X|_{X_k}$ to the $k$-th infinitesimal neighborhood $X_k$ of $C$ in $X$ 
is ample \cite[Prop.4.2]{Har70}.
Thus there are projective schemes $Z_k$ such that $(Z_k)_{\rm an} \simeq X_k$. 
Then the formal scheme arising as inverse limit $Z$ of the system $(Z_k)$ is the algebraization of $\hat X$. 

 
Following the notations of Koll\'ar \cite[Constr.1.2]{Kol91b}, choose $d \in \Z$ with $0 \leq d < m$ such that 
$$ d \equiv -mK_X \cdot C \ {\rm mod}(m). $$ 
Choose a general Cartier divisor $D$ on $\hat X$ such that $D$ meets $C$ in a single point  $p_0$ transversally. By the general choice of $D$ this
will be a smooth point of $\hat X.$ Therefore the pull-back
$$ 
\nu^* \mathcal O_{\hat X}(-dD-mK_{\hat X})
$$ 
to the normalisation $\holom{\nu}{\tilde C}{C}$ 
is divisible by $m$ in ${\rm Pic}(\tilde C).$ 
We claim that 
$$ \mathcal O_{\hat X}(-dD-mK_{\hat X})$$
is divisible by $m$ in ${\rm Pic}(\hat X)$. To see this observe that for each $k \in \N$ the kernel 
of the restriction 
$$ \mathrm {Pic}(X_{k+1}) \to \mathrm{Pic}(X_k)$$
is an extension by an additive and a multiplicative group 
(use the exponential sequence for $X_k$, cp. \cite[1.2.1]{Kol91b} and \cite[no.232,Prop.6.5]{Gro62}).  
These groups are divisible, so $$ \mathcal O_{X_k}(-dD-mK_{\hat X})$$
is divisible by $m$ in ${\rm Pic}(X_k)$, hence $ \mathcal O_{\hat X}(-dD-mK_{\hat X})$ is divisible by $m$ in ${\rm Pic}(\hat X)$ \cite[II, Ex.9.6]{Har77}.
\\
Write
$$ \mathcal O_{\hat X}(-dD-mK_{\hat X}) \simeq N^{\otimes m}$$
with a line bundle $N$ on $\hat X$. 
We obtain a section $$s \in H^0(\hat X, (N^*)^{\otimes m} \otimes \mathcal O_{\hat X}(-mK_{\hat X}))$$
via $$N^{\otimes m} \otimes \mathcal O_{\hat X}(mK_{\hat X}) \simeq \mathcal O_{\hat X}(-dD)  \to \mathcal O_{\hat X}.$$ 
Let $$ p: \overline X \to \hat X$$
be the normalisation of $\hat X[s^{1 \over m}].$ By construction $p$ ramifies exactly over $D$ and the non-Gorenstein points of $\hat X$. 
Moreover $\overline X$ is Gorenstein and has at most terminal singularities by \cite[5.21(4)]{KM98}. 
We introduce the number
$$ e := {m \over {{\rm gcd}(d,m)}}.$$
Then $p$ has ramification order $e$ over $D$ yielding the equation
\begin{equation} \label{equationstar}
K_{\overline X} = p^*(K_{\hat X} + (1 - {1 \over {e}})D). 
\end{equation} 
Let $\overline C  \subset p^{-1}(C)$ be an irreducible component (with reduced structure). 
We claim that if $K_{\hat X} \cdot C < -1,$ then $K_{\hat X} \cdot \overline C \leq -2$. 
Assuming this for the time being, let us see how to conclude:
since
$$
K_{\hat X} \cdot C = K_X \cdot C < -1,
$$ 
we know by the claim that $K_{\overline X} \cdot \overline C \leq -2$.
But now $\overline X$ is a threefold with Gorenstein terminal singularities, so it has at most lci singularities. Therefore $\overline C$ 
deforms in $\overline X$ by \cite[Thm.1.15, Rem.1.17] {Kol96}. Hence some multiple of $C$ deforms in $\hat X$.
 
{\em Proof of the claim.}
Since $K_{\hat X} \cdot C$ is an integer multiple of ${{1} \over {e}}$,  our assumption implies
$K_{\hat X} \cdot C \leq -1 - {{1} \over {e}}$. Since $D \cdot C = 1$ we obtain
$$
K_{\hat X} \cdot C + (1 - {1 \over {e}}) D \cdot C \leq - \frac{2}{e}
$$
The order of ramification for every point $p_0$ over $D$ is equal to $e$, so we have $\deg(\overline C/C) \geq e$.
Thus the ramification formula \eqref{equationstar} and the preceding inequality gives
$$ 
K_{\overline X} \cdot \overline C = \deg (\overline C/C) \left( K_{\hat X} \cdot C + (1 - {1 \over {e}}) D \cdot C \right)  \leq -2.
$$
This proves the claim. 

{\em Step 2. Very rigid curves $K_X$-negative curves are rational.} 
If $C$ is irrational and $K_X \cdot C<0$, then there are finite \'etale covers
$$ h: C' \to C$$ 
of arbitrary large degree. Choose an open neighborhood $U$ of $C$ in $X$ such that $C$ is a deformation retract of $U$. Then 
$\pi_1(U) \simeq \pi_1(C)$ and therefore 
we obtain a finite \'etale cover  $g: U' \to U$ such that $C' \subset U'$ and $h = g|_{C'}$. Hence $K_{U'} \cdot C'$ gets arbitrarily negative, in 
particular smaller than $-1$. By Step 1 the curve $C'$ is not very rigid, hence $C$ is not very rigid.
\end{proof} 

\section{Bend and break} \label{sectionBB}

For the convenience of the reader we recall two results from the deformation theory of curves. 

\begin{theorem} \label{theoremdeformationssurface} \cite[Thm.1.15]{Kol96}
Let $S$ be a smooth complex surface, and let $C \subset S$ be a curve. Then we have
$$
\dim_C \Chow{S} \geq -K_S \cdot C + \chi(\sO_{\tilde C}),
$$ 
where $\tilde C \rightarrow C$ is the normalisation.
\end{theorem}

Recall that  terminal Gorenstein threefold singularities are hypersurfaces singularities \cite{Rei80}, \cite[Cor.5.38]{KM98}, so we have:  

\begin{theorem} \label{theoremdeformations} \cite[Thm.1.15]{Kol96}
Let $X$ be a normal $\Q$-factorial complex threefold with at most terminal {\em Gorenstein}  singularities. Let $C \subset X$ be a curve. Then we have
$$
\dim_C \Chow{X} \geq -K_X \cdot C.
$$ 
\end{theorem}

\begin{definition} \label{definitionsplit}
Let $X$ be a complex space and $Z$ an effective $1$-cycle on $X$. A deformation family of $Z$ is a
family of cycles $(Z_t)_{t \in T}$ where $T$ is an irreducible component $T$ of the cycle space $\Chow{X}$ that contains the point corresponding to $Z$. 
We say that the family of deformations splits if there exists a $t_0 \in T$ such that the cycle 
$$
Z_{t_0} = \sum \alpha_l C_l
$$
has reducible support or $Z_{t_0} = \alpha_1 C_1$ with $\alpha_1 \geq 2$.
\end{definition}

The following elementary lemma will be used several times.

\begin{lemma} \label{lemmabasic}
Let $X$ be a  normal  $\Q$-factorial compact K\"ahler threefold with at most terminal singularities
such that $K_X$ is pseudoeffective. 
Let $C \subset X$ be a curve such that
$K_X \cdot C<0$ and $\dim_C \Chow{X}>0$.

Then there exists a unique surface $S_j$ from the divisorial Zariski decomposition \eqref{Bdecomposition}
such that $C$ and its deformations are contained in the surface $S_j$.
Moreover we have
\begin{equation} \label{canonicaldegree}
K_{S_j} \cdot C < K_X \cdot C.
\end{equation}
\end{lemma}

\begin{proof}
Let $(C_t)_{t \in T}$ be a deformation family of $C$. Since $C$ is integral, a general deformation $C_t$ is
integral. Since $K_X \cdot C_t<0$ for all $t \in T$ we see that the locus
$$
S := \overline{\cup_{t \in T} C_t}
$$
has the property that $K_X|_S$ is not pseudoeffective. 
Since $T$ and the general $C_t$ are irreducible, the
graph of the family $\mathcal C \rightarrow T$ is irreducible.
Thus $S$ is irreducible, and since $K_X$ is pseudoeffective, it is a surface.
By Lemma \ref{lemmasurfaces} there exists a $j \in \{ 1, \ldots, r \}$ such that $S=S_j$.

In order to prove the uniqueness of the surface $S_j$ and the inequality \eqref{canonicaldegree} 
it is sufficient to show that $S_j \cdot C < 0$.
The deformation family $(C_t)_{t \in T}$ has no fixed component, in particular for every $k \in \{ 1, \ldots, r \}$ 
such that $k \neq j$ there exists a $t \in T$ general such that
$C_t \not\subset S_k$. 
Thus we have $ S_k \cdot C = S_k \cdot C_t \geq 0$ for every $k \neq j$. 
Moreover the restriction $N(K_X)|_{S_j}$ is pseudoeffective and the family $(C_t)_{t \in T}$ covers $S_j$, so 
$$
N(K_X) \cdot C = N(K_X)|_{S_j} \cdot C = N(K_X)|_{S_j} \cdot C_t \geq 0.
$$
Since
$$
0 > K_X \cdot C =\sum_{j=1}^r \lambda_j S_j \cdot C + N(K_X) \cdot C,
$$ 
this implies that $S_j \cdot C<0$. This shows also that {\em every} deformation family of $C$ lies in $S_j$.
Thus the surface $S_j$ does not depend on the choice of $T$.
\end{proof}

\begin{lemma} \label{lemmabb}
Let $\hat S$ be a smooth projective surface that is uniruled, and let
$\hat C \subset \hat S$ be an irreducible curve.

a) Suppose that $K_{\hat S} \cdot \hat C <0$.
Then there exists an effective 1-cycle $\sum \alpha_k C_k$ with coefficients in $\Q^+$ such that
$$
[\sum \alpha_k C_k] = [\hat C]
$$
and $C_1$ is a rational curve such that $K_{\hat S} \cdot C_1 <0$.

b) Suppose that $K_{\hat S} \cdot \hat C \leq -4$.
Then there exists an effective 1-cycle $\sum_{k=1}^m C_k$ with $m \geq 2$  
such that
$$
[\sum_{k=1}^m C_k] = [\hat C]
$$
such that $K_{\hat S} \cdot C_1 <0$ and $K_{\hat S} \cdot C_2 <0$.
\end{lemma}

\begin{proof}
The statements are trivial for $\hat S \simeq \PP^2$, so suppose that this is not the case.
Fix an ample line bundle $H$ on $\hat S$. We will argue by induction on the degree $d:=H \cdot \hat C$.
We start the induction with $d=0$ where both statements are trivial.
Let us now do the induction step.

{\em Proof of a)}
If $\hat C$ is a rational curve we are finished. 
If $\hat C$ is not rational we know by Theorem \ref{theoremdeformationssurface} that $\dim_{\hat C} \Chow{S} \geq 1$,
so we have a positive-dimensional family of deformations $(C_t)_{t \in T}$.

{\em 1st case. The family of deformations splits.} Thus there exists a cycle $\sum \beta_l C_l$
that is not integral such that $[\sum \beta_l C_l]=[\hat C]$. Up to renumbering we can suppose
that $K_{\hat S} \cdot C_1<0$.  Since $H \cdot C_1<H \cdot \hat C$ we can apply the induction
hypothesis to $C_1$ and conclude.

{\em 2nd case. The family of deformations does not split.}
We claim that in this case $\hat S$ is a ruled surface or $\PP^2$. 
Arguing by contradiction, let \holom{\sigma}{\hat S}{S_0} be a MMP, i.e., a sequence of blow-downs of $(-1)-$curves,  to some Mori fibre space $S_0$
which we can suppose to be a ruled surface\footnote{If $S_0$ is $\PP^2$ just take the same MMP but omit the
last blow-up.}. Let $E$ be the $\sigma$-exceptional locus, then
$\hat C \cdot E>0$ since the deformations of $\hat C$ cover $\hat S$ and do not split.
Thus the family $(\sigma(C_t))_{t \in T}$ is positive-dimensional, does not split and has a fixed point
$p \in \sigma(E)$. This contradicts \cite[Lemma 3.3]{Pet01}.

Since by our assumption $\hat S \not\simeq \PP^2$, this shows that $\hat S$ is a ruled surface. 
We claim that $[\hat C]$ lies in the interior of the Mori cone $\NE{\hat S}$: otherwise 
$\NE{\hat S}$ would be generated by  $[\hat C]$ and $[F]$ where $F$ is a fibre of the ruling\footnote{Note
that $[\hat C] \not\in \R^+ [F]$ since $\hat C$ is not rational}.
Since $K_{\hat S} \cdot \hat C<0$ we obtain that $\hat S$ is a del Pezzo surface, hence $\mathbb F_1$ or $\PP^1 \times \PP^1$.
In both cases every curve in the extremal ray $\R^+ [\hat C]$ is rational, a contradiction.

Thus we can choose a curve $B \subset \hat S$ such that $[C]$ lies
in the interior of the cone generated by $[B]$ and $[F]$, where $F$ is a fibre of the ruling. Thus we have
$$
[C] = \lambda_1 [F] + \lambda_2 [B],
$$
with $\lambda_i \in \Q^+$.

{\em Proof of b).} 
By Theorem \ref{theoremdeformationssurface} we have $\dim_{\hat C} \Chow{S} \geq 2$,
so we have a two-dimensional family of deformations $(C_t)_{t \in T}$.
Note that the family of deformations splits: arguing by contradiction, let \holom{\sigma}{\hat S}{S_0} be a MMP to some Mori fibre space $S_0$
which we can suppose to be a ruled surface. Since $T$ has dimension at least two, the subfamily $T_p$
parametrising the deformations through a general point $p \in \hat S$ has dimension at least one.
Thus $(\sigma(C_t))_{t \in T_p}$ is positive-dimensional, does not split and has a fixed point
$p \in S_0$. This contradicts \cite[Lemma 3.3]{Pet01}.

Since the deformations split, there exists a cycle $\sum_{l=1}^{m'} C_l$
with $m' \geq 2$ such that $[\sum C_l]=[\hat C]$. Up to renumbering we can suppose
$$
K_{\hat S} \cdot C_1 \leq K_{\hat S} \cdot C_2 \leq \ldots \leq K_{\hat S} \cdot C_{m'}.
$$
If $K_{\hat S} \cdot C_2<0$ we are finished. If not we have
$K_{\hat S} \cdot C_1 \leq K_{\hat S} \cdot \hat C$.
Since $H \cdot C_1<H \cdot \hat C$ we can apply the induction
hypothesis to $C_1$ to conclude.
\end{proof}

\begin{lemma} \label{lemmadeform}
Let $X$ be a  normal  $\Q$-factorial compact K\"ahler threefold with at most terminal singularities
such that $K_X$ is pseudoeffective. Let $S_1, \ldots, S_r$ be the surfaces appearing in the divisorial Zariski decomposition \eqref{Bdecomposition}.
Set
$$
b := \max \{ 1, -K_X \cdot Z \ | \ Z \ \mbox{a curve s.t.} \ Z \subset S_{j, \sing} \ \mbox{or} \ Z \subset S_{j} \cap S_{j'}  \}.
$$
If $C \subset X$ is a curve
such that
$$
-K_X \cdot C > b,
$$
then we have $\dim_C \Chow{X}>0$.
\end{lemma}

\begin{proof}
Since $b \geq 1$, the curve $C$ is not very rigid
by Theorem \ref{theoremveryrigid}.
Let $m \in \N$ be minimal such that $\dim_{mC} \Chow{X} > 0$, and let $(C_t)_{t \in T}$ be a family of deformations.
Let $\mathcal C \rightarrow T$ be the graph of the family: up to replacing $\mathcal C$ by an 
irreducible component $\mathcal C' \subset \mathcal C$ that contains $C$ we can suppose that
$\mathcal C$ is irreducible. In particular, the locus covered by the effective 1-cycles $C_t$ is irreducible.
The family $(C_t)_{t \in T}$ has no fixed component, i.e. there does not exist a curve $B$ that is contained
in the support of every $C_t$. Indeed, since $mC$ is a member of the family, we would have $B=C$, but then
$\dim_{(m-1)C} \Chow{X}>0$, contradicting the minimality of $m$. Thus if $C_t$ is general and
$$
C_t = \sum_l \alpha_l C_{t,l}
$$
its decomposition, we have $\dim_{C_{t,l}} \Chow{X}>0$ for all $l$. Up to renumbering we can suppose that
$K_X \cdot C_{t,1}<0$. By Lemma \ref{lemmabasic} applied to $C_{t,1}$ there exists a unique surface 
$S_j$ from the decomposition \eqref{Bdecomposition} such that 
$$
\overline{\bigcup_{t \in T \ \mbox{general}} C_{t,1}} = S_j.
$$
Since the locus covered by the family $(C_t)_{t \in T}$ is irreducible, we see that $C \subset S_j$.

By the definition of $b$, we have $C \not\subset S_l$ for every $l \neq j$, 
thus 
$$
S_l \cdot C \geq 0 \qquad \forall \ l \neq j.
$$
The restriction $N(K_X)|_{S_j}$ is pseudoeffective, and for $t \in T$ general all the curves $C_{t,l}$ are movable
in $S_j$, so we get
$$
N(K_X) \cdot C = N(K_X)|_{S_j} \cdot C = N(K_X)|_{S_j}  \cdot C_t = \sum_l \alpha_l (N(K_X)|_{S_j}  \cdot C_{t,l}) \geq 0.
$$
Since $K_X \cdot C<0$, the equality \eqref{Bdecomposition} now implies $S_j \cdot C < 0$. Thus we have
$$
K_{S_j} \cdot C < K_X \cdot C < -b.
$$
By definition of $b$, the curve $C$ is not contained in the singular locus of $S_j$.
Let \holom{\pi_j}{\hat{S}_j}{S_j} be the composition of normalisation and minimal resolution (cf. Subsection \ref{subsectionsurfaces}). 
The strict transform
$\hat C$ of $C$ is well-defined and by \eqref{inequalitycanonical} we have  
$$
K_{\hat S_j} \cdot \hat C \leq K_{S_j} \cdot C < -b.
$$
Since $b \geq 1$ we obtain by Theorem \ref{theoremdeformationssurface} that
$$
\dim_{\hat C} \Chow{\hat S}>0,
$$
so $\hat C$ deforms. Thus its push-forward $\pi_* \hat C = C$ deforms.
\end{proof}

\begin{corollary} \label{corollarybreak}
Let $X$ be a  normal  $\Q$-factorial compact K\"ahler threefold with at most terminal singularities
such that $K_X$ is pseudoeffective. Let $b$ be the constant from Lemma \ref{lemmadeform} and set
$$
d := \max \{ 3, b \}.
$$
If $C \subset X$ is a curve
such that $-K_X \cdot C > d$,
then we have
$$
[C] = [C_1] + [C_2]
$$
with $C_1$ and $C_2$ effective 1-cycles (with integer coefficients) on $X$.
\end{corollary}

\begin{proof}
Since $d \geq b$ 
we know by Lemma  \ref{lemmadeform} that $\dim_{C} \Chow{X} > 0$. 
Let $(C_t)_{t \in T}$ be a family of deformations.
By Lemma \ref{lemmabasic},
there exists a unique surface $S_j$ from the divisorial Zariski decomposition \eqref{Bdecomposition}
such that the $C_t$ are contained in the surface $S_j$.
Moreover we have
$$
K_{S_j} \cdot C < K_X \cdot C < d.
$$
By the definition of the constant $b$ in Lemma \ref{lemmadeform}, we have
$C \not\subset S_{j, \sing}$.

Let \holom{\pi_j}{\hat{S}_j}{S_j} be the composition of normalisation and minimal resolution (cf. Subsection \ref{subsectionsurfaces}). 
The strict transform
$\hat C$ of $C$ is well-defined and by \eqref{inequalitycanonical} we have  
$$
K_{\hat S_j} \cdot \hat C \leq K_{S_j} \cdot C < -3.
$$
Thus by Lemma \ref{lemmabb},b) there exists an effective 1-cycle $\sum_{k=1}^m C_k$ with $m \geq 2$  
such that
$$
[\sum_{k=1}^m C_k] = [\hat C]
$$
such that $K_{\hat S_j} \cdot C_1 <0$ and $K_{\hat S_j} \cdot C_2 <0$. 
Since $K_{\hat S_j}$ is $\pi_j$-nef, we have $(\pi_j)_* C_1 \neq 0$ and $(\pi_j)_* C_2 \neq 0$. 
Thus we obtain 
$$
[C] =  (\pi_j)_* [\hat C] = \sum_{k=1}^m (\pi_j)_* [C_k]
$$ 
and the first two terms of this sum are not zero.
\end{proof}

\begin{lemma} \label{lemmarationalrepresentative}
Let $X$ be a  normal  $\Q$-factorial compact K\"ahler threefold with at most terminal singularities
such that $K_X$ is pseudoeffective. 
Let $\R^+ [\Gamma_i]$ be a $K_X$-negative extremal ray in $\NEX$ where
$\Gamma_i$ is a curve that is not very rigid. Then the following holds:
\begin{enumerate}
\item There exists a curve $C \subset X$ such that $[C] \in \R^+ [\Gamma_i]$ and $\dim_C \Chow{X}>0$.
\item There exists a rational curve $C \subset X$ such that $[C] \in \R^+ [\Gamma_i]$.
\end{enumerate}
\end{lemma}

\begin{proof}
Let 
$$
m := \min \{  k \in \N \ | \ \exists \ C \subset X \ \mbox{curve s.t.} \ [C] \in \R^+ [\Gamma_i],  \dim_{k C} \Chow{X}>0 \}.
$$
Our goal is to show that $m=1$. Arguing by contradiction let $C$ be a curve that realises the minimal degree $m>1$.
We have $\dim_{mC} \Chow{X}>0$, so let $(C_t)_{t \in T}$ be a family of deformations.
Note that this family has no fixed component, since we chose $m$ to be minimal.
Thus if $C_t$ 
is a general member of the family, then for every irreducible component $C_t' \subset C_t$ we
have $\dim_{C_t'} \Chow{X}>0$. Since the ray $\R^+ [\Gamma_i]$ is extremal, we have $C_t' \in \R^+ [\Gamma_i]$,
a contradiction.

For the proof of statement b), choose $C \in \R^+ [\Gamma_i]$ such that $\dim_{C} \Chow{X}>0$.
By Lemma \ref{lemmabasic} the deformations of $C$ are contained in
one of the surfaces $S_j$ from the divisorial Zariski decomposition \eqref{Bdecomposition} and we have
$$
K_{S_j} \cdot C < K_X \cdot C <0.
$$
Up to replacing $C$ by a general deformation we can suppose that $C \not\subset S_{j, sing}$.
Let \holom{\pi_j}{\hat{S}_j}{S_j} be the composition of normalisation and minimal resolution (cf. Subsection \ref{subsectionsurfaces}). 
The strict transform
$\hat C$ of $C$ is well-defined and by \eqref{inequalitycanonical} we have  
$$
K_{\hat S_j} \cdot \hat C \leq K_{S_j} \cdot C <0.
$$
By Lemma \ref{lemmabb},a) there exists an effective 1-cycle $\sum \alpha_k C_k$ with coefficients in $\Q^+$ such that
$$
[\sum \alpha_k C_k] = [\hat C]
$$
and $C_1$ is a rational curve such that $K_{\hat S_j} \cdot C_1 <0$. 
Since $K_{\hat S_j}$ is $\pi_j$-nef, we have $(\pi_j)_* C_1 \neq 0$. 
Thus we obtain 
$$
[C] = (\pi_j)_* [\hat C] = \sum_k \alpha_k (\pi_j)_* [C_k],
$$ 
and the first term $\alpha_1 (\pi_j)_* [C_1]$ is not zero. Since the ray $\R^+ [\Gamma_i]$
is extremal, we see that $[C_1]$ is a positive multiple of $[\Gamma_i]$.
\end{proof}

\section{Cone theorems}

\begin{lemma} \label{lemmaclosure}
Let $X$ be a normal compact K\"ahler threefold such that $K_X$ is $\Q$-Cartier.
Let $N \subset \NAX \subset N_1(X)$ be a closed convex cone.
Suppose that there exists a number $d \in \N$ and a countable family $(\Gamma_i)_{i \in I}$  of curves  on $X$
such that 
$$
0 < -K_X \cdot \Gamma_i \leq d
$$
and
$$
N = \overline{N_{K_X \geq 0} + \sum_{i \in I} \R^+ [\Gamma_i]},
$$ 
Then we have
$$
N = N_{K_X \geq 0} + \sum_{i \in I} \R^+ [\Gamma_i].
$$ 
\end{lemma}

Our argument follows the proof of the cone theorem for projective manifolds in \cite{Deb01}.

\begin{proof}
Set $V:=N_{K_X \geq 0} + \sum_i \R^+ [\Gamma_i]$. We want to prove that $\overline V=V$, i.e.
the cone $V$ is closed.
By assumption $N \subset \NAX$, so we have $V \subset \NAX$. In particular $V$ contains no lines 
and is the convex hull of its extremal rays \cite[Lemma 6.7 b)]{Deb01}.
It is thus sufficient to show that if $\R^+ [r]$ is an extremal ray in $\overline V$ such that
$K_X \cdot r<0$, then $\R^+ [r]$ is in $V$.

Let $\omega$ be a K\"ahler form on $X$ and choose a $\varepsilon>0$ such that
$$
(K_X+\varepsilon \omega) \cdot r < 0.
$$
Then we have by our hypothesis
$$
N = \overline{N_{K_X+\varepsilon \omega \geq 0} + \sum_{j \in J} \R^+ [\Gamma_j]}
$$ 
where the sum runs over those indices $j \in I$ such that $(K_X+\varepsilon \omega) \cdot \Gamma_j<0$.
Since $-K_X \cdot \Gamma_j \leq d$ we have $\omega \cdot \Gamma_j \leq \frac{d}{\varepsilon}$. Since the degree
of the cohomology classes $[\Gamma_j]$ with respect to $\omega$ is bounded,  there are only finitely many
such classes. In particular we have
$$
N = N_{K_X+\varepsilon \omega \geq 0} + \sum_{j \in J} \R^+ [\Gamma_j]
$$ 
and up to renumbering we can suppose that $J = \{ 1, \ldots, q \}$ for some $q \in \N$.

Write now $r$ as the limit of a sequence $r_m+s_m$ such that for all $m \in \N$ one has
$$
(K_X+\varepsilon \omega) \cdot r_m \geq 0
$$
and $s_m \in \sum_{j=1}^q  \R^+ [\Gamma_j]$, that is 
$$
s_m = \sum_{j=1}^q \lambda_{j, m} \Gamma_j.
$$
Since $\omega \cdot (r_m+s_m)$ converges to $\omega \cdot r$ 
and $\omega$ is non-negative on every element of $N$, the sequences
$\omega \cdot r_m$ and $\omega \cdot \lambda_{j, m} \Gamma_j$ are bounded
by $\omega \cdot r+1$ for large $m$. In particular we can assume after taking subsequences
that the sequences $r_m$ and $\lambda_{j, m} \Gamma_j$ converge. Since
$$
r = \lim_{m \to \infty} r_m + \sum_{j=1}^q \lim_{m \to \infty} \lambda_{j, m} \Gamma_j.
$$
and $r$ is extremal in $\overline V$ this implies that  $\lim_{m \to \infty} r_m$  and $\lim_{m \to \infty} \lambda_{j, m} \Gamma_j$
are non-negative multiples of $r$. 
Since $(K_X+\varepsilon \omega) \cdot \lim_{m \to \infty} r_m \geq 0$ and
$(K_X+\varepsilon \omega) \cdot r<0$ we get that $\lim_{m \to \infty} r_m=0$. 
This implies that there exists a $j_0 \in \{ 1, \ldots, q \}$ such that $\lim_{m \to \infty} \lambda_{j_0, m} z_{j_0}$ 
is a positive multiple of $r$. Hence
the extremal rays $\R^+ r$ and $\R^+ z_{j_0}$ coincide, so $\R^+ r$ is in $V$.
\end{proof}

\begin{theorem} \label{theoremNE}
Let $X$ be a  normal  $\Q$-factorial compact K\"ahler threefold with at most terminal singularities
such that $K_X$ is pseudoeffective. 
Then there exists a number $d \in \N$ and a countable family $(\Gamma_i)_{i \in I}$  of curves  on $X$
such that 
$$
0 < -K_X \cdot \Gamma_i \leq d
$$
and
$$
\NEX = \NEX_{K_X \geq 0} + \sum_{i \in I} \R^+ [\Gamma_i].
$$
If the ray $\R^+ [\Gamma_i]$ is extremal in $\NEX$, there exists a rational curve $C_i$ on $X$ 
such that $[C_i] \in \R^+ [\Gamma_i]$.
\end{theorem}

\begin{proof}[Proof of Theorem \ref{theoremNE}]
Let $d \in \N$ be the bound from Corollary \ref{corollarybreak}. There are only countably many curve classes $[C] \in \NEX$,
and we choose a representative $\Gamma_i$ for each class such that  $0 < -K_X \cdot \Gamma_i \leq d$.
We set
$$
V := \NEX_{K_X \geq 0} + \sum_{0 < -K_X \cdot \Gamma_i \leq d} \R^+ [\Gamma_i].
$$
Fix a K\"ahler form $\omega$ on $X$ such that 
$$
\omega \cdot C \geq 1
$$ 
for every curve $C \subset X$\footnote{
The cohomology class of a curve $C \subset X$ is contained in the integral
lattice $H^4(X, \Z)$, so
the cohomology classes of curves in $X$ form a discrete set in $\NEX$. Thus for every 
K\"ahler form $\omega'$ there exists a real constant $\lambda>0$ such that
$$   
\omega' \cdot C \geq \lambda
$$ 
for every curve $C \subset X$. 
}. 

{\em Step 1. We have $\NEX=V$.} 
By Lemma \ref{lemmaclosure} it is sufficient to prove that $\NEX = \overline V$, i.e.
the class $[C]$ of every irreducible curve $C \subset X$
is contained in $V$. 
We will prove the statement by induction on the degree $l:=\omega \cdot C$. The start
of the induction for $l=0$ is trivial.
Suppose now that we have shown the statement for all curves of degree at most $l-1$
and let $C$ be a curve such that $l-1 < \omega \cdot C \leq l$.
If $-K_X \cdot C \leq d$ we are done. Otherwise there exists by Corollary \ref{corollarybreak} 
a decomposition
$$
[C] = [C_1] + [C_2]
$$
with $C_1$ and $C_2$ effective 1-cycles (with integer coefficients) on $X$.
Since $\omega \cdot C_i \geq 1$ for $i=1,2$ we have
$\omega \cdot C_i \leq l-1$ for $i=1,2$. By induction both classes are in $V$, so $[C]$ is in $V$.

{\em Step 2. Every extremal ray contains the class of a rational curve.}
If the ray $\R^+ [\Gamma_i]$ is extremal in $\NEX$ we know by Theorem \ref{theoremveryrigid}
and Lemma \ref{lemmarationalrepresentative} that there exists a rational curve $C_i$ such that $[C_i]$
is in the extremal ray.
\end{proof}

\begin{theorem} \label{theoremNApreliminary}
Let $X$ be a  normal  $\Q$-factorial compact K\"ahler threefold with at most terminal singularities
such that $K_X$ is pseudoeffective. 
Then there exists a $d \in \N$ and a countable family $(\Gamma_i)_{i \in I}$  of curves  on $X$
such that 
$$
0 < -K_X \cdot \Gamma_i \leq d
$$
and
$$
\NAX = \NAX_{K_X \geq 0} + \sum_{i \in I} \R^+ [\Gamma_i] 
$$
If the ray $\R^+ [\Gamma_i]$ is extremal in $\NAX$, there exists a rational curve $C_i$ on $X$ 
such that $[C_i] \in \R^+ [\Gamma_i]$.
\end{theorem}

Theorem \ref{theoremNApreliminary} is a consequence of Theorem \ref{theoremNE} and the following proposition.

\begin{proposition} \label{propositionNA}
Let $X$ be a  normal  $\Q$-factorial compact K\"ahler threefold with at most terminal singularities
such that $K_X$ is pseudoeffective. 

Suppose that there exists a $d \in \N$ and a countable family $(\Gamma_i)_{i \in I}$  of curves  on $X$
such that 
$$
0 < -K_X \cdot \Gamma_i \leq d
$$
and
$$
\NEX = \NEX_{K_X \geq 0} + \sum_{i \in I} \R^+ [\Gamma_i] 
$$
Then we have
$$
\NAX = \NAX_{K_X \geq 0} + \sum_{i \in I} \R^+ [\Gamma_i] 
$$
\end{proposition}

\begin{proof}
Set
$$
V := \NAX_{K_X \geq 0} + \sum_{i \in I} \R^+ [\Gamma_i].
$$
By Lemma \ref{lemmaclosure} it is sufficient to show that $\NAX \subset \overline{V}$.
Let $\pi: \hat X \to X$ be a desingularisation and consider $\alpha \in \NAX.$ By Proposition \ref{ext}, there exists $\hat \alpha \in \overline {NA}(\hat X)$ such that 
$\alpha = \pi_*(\hat \alpha). $ 
By \cite[Cor.0.3]{DP04} the cone $\overline {NA}(\hat X)$ is the closure of the convex cone generated by cohomology classes
of the form $[\hat \omega]^2, [\hat \omega] \cdot [\hat S]$ and $[\hat C]$ where $\hat \omega$ is a K\"ahler form,
$\hat S$ a surface and $\hat C$ a curve on $\hat X.$ 
Our goal is to show that the $\pi-$image of any of this three types is contained in $\overline V$. 

{\em 1st case. $\hat \alpha = [\hat \omega]^2$ with $\hat \omega$ a K\"ahler form.} 
Since
$\pi^*(K_{X})$ is pseudoeffective, we have $\pi^*(K_{X}) \cdot [\hat \omega]^2 \geq 0$, hence 
$$ K_X \cdot \alpha = K_X \cdot \pi_*(\hat \alpha) = \pi^*(K_{X}) \cdot \hat \alpha \geq 0, $$
and thus 
$\alpha \in \NAX_{K_X \geq 0}$. 

{\em 2nd case. $\hat \alpha = [\hat C]$ with $\hat C$ a curve.} 
Then set $C := \pi_*(\hat C)$, so that $\alpha = [C].$
Since we have an inclusion
\begin{equation} \label{inclusion}
\NEX = \NEX_{K_X \geq 0} + \sum_{i \in I} \R^+ [\Gamma_i] \hookrightarrow \NAX_{K_X \geq 0} + \sum_{i \in I} \R^+ [\Gamma_i], 
\end{equation}
and by hypothesis any curve class $[C]$ is in the left hand side, we see that $[C] \in V$.\\
(3) 
{\em 3rd case. $\hat \alpha = [\hat \omega] \cdot [\hat S]$ with $\hat S$ an irreducible surface
and $\hat \omega$ a K\"ahler form.}
If 
$$\pi^*(K_X) \cdot [\hat \omega] \cdot [\hat S] \geq 0,$$
the class $\alpha $ is in $\NAX_{K_X \geq 0}$.
Suppose now that 
$$
\pi^*(K_X )\cdot [\hat \omega] \cdot [\hat S] < 0.
$$ 
{\it We claim} that $\hat \alpha \in \NE{\hat X}$. Since $\pi_*(\NE{\hat X})=\NEX$ this implies that
$\alpha=\pi_* (\hat \alpha)$
is in the image of the inclusion \eqref{inclusion} which concludes the proof.

{\em Proof of the claim.} 
Using the projection formula we see that $\pi(\hat S)$ is not a point.
Since $X$ has at most isolated singularities we see that $\pi(\hat S)$ is not a curve, 
so we can suppose that $\pi(\hat S)$ is a surface $S$.
Since $\pi^*(K_X) \cdot [\hat \omega] \cdot [S] < 0$, the restriction $\pi^*(K_X)|_{\hat S}$ is not pseudoeffective. 
Thus the restriction $K_X \vert_S$ is not pseudoeffective. By Lemma \ref{lemmasurfaces} 
the surface $S$ is one of the surfaces $S_j$. In particular it is uniruled, hence $\hat S$ is uniruled.
Let $\pi_j: \hat S_j \rightarrow \hat S \subset \hat X$ be the composition of normalisation and minimal resolution (cf. Subsection \ref{subsectionsurfaces}). Then $\hat S_j$ is uniruled and projective. 
Since $\hat S_j$ is uniruled, we have $H^2(\hat S_j, \sO_{\hat S_j})=0$. In particular the Chern class map
$$
\pic{\hat S_j} \rightarrow H^2(\hat S_j, \Z)
$$ 
is surjective, so
$$
NS_{\R}(\hat S_j) = H^2(\hat S_j, \R) \cap H^{1,1}(\hat S_j).
$$
Thus the cohomology class $[\pi_j^* \hat \omega]$ is represented by an $\R$-divisor
which is nef and big. In particular it is a limit of classes of ample $\Q$-divisors $H_m$ on $\hat S_j$ which we can
represent by classes of effective 1-cycles $[C_m]$ with rational coefficients.
We claim that the sequence
$$
(\pi_j)_* [C_m] \in \NE{\hat X} \subset N_1(\hat X)
$$
converges to $\hat \alpha$. 
By duality it is sufficient 
to prove that for every $\eta$ a real closed $(1,1)$-form on $\hat X$, the sequence
$[\eta] \cdot (\pi_j)_* [C_m]$ converges to $[\eta] \cdot [\hat \omega] \cdot [\hat S]$.
Yet by the projection formula we have
$$
[\eta] \cdot (\pi_j)_* [C_m] = [(\pi_j)^* \eta] \cdot [C_m]
$$
and
$$
[\eta] \cdot [\hat \omega] \cdot [\hat S] = \int_{\hat S} \eta|_{\hat S} \wedge \hat \omega|_{\hat S} = 
\int_{\hat S_j} (\pi_j)^* \eta \wedge (\pi_j)^* \omega =
[(\pi_j)^* \eta] \cdot [(\pi_j)^* \omega],  
$$
so the convergence follows from the construction of the sequence $[C_m]$.
\end{proof}

\begin{proof}[Proof of Theorem \ref{theoremNA}]
The only statement that is not part of Theorem \ref{theoremNApreliminary} is that in every $K_X$-negative extremal ray $R_i$ 
we can find a rational curve $\Gamma_i$ such that $\Gamma_i \in R_i$ and
$$
0 < -K_X \cdot \Gamma_i \leq 4.
$$ 
If the extremal ray is small, this is clear by Theorem \ref{theoremveryrigid}.
If the extremal ray is divisorial, the contraction exists by Theorem \ref{theoremcontraction}\footnote{The proof of Theorem 
\ref{theoremcontraction} uses only Theorem \ref{theoremNApreliminary}.}. 
Thus we can conclude by \cite[Thm.7.46]{Deb01}.
\end{proof}

\begin{remark}
The cone theorem \ref{theoremNA} actually holds if $X$ has at most canonical singularities: using a relative MMP (which exists
since we can take a resolution of singularities which is a projective morphism) we construct a bimeromorphic 
morphism \holom{\mu}{X'}{X} such that $K_X' \simeq \mu^* K_X$ and $X'$ has at most terminal singularities.
Since $\mu_* (\NA{X'})=\NAX$ by Proposition \ref{ext}, the statement follows from our theorem.
\end{remark}

\section{Contractions of extremal rays}

For the whole section we make the following

\begin{assumption*}
Let $X$ be a  normal  $\Q$-factorial compact K\"ahler threefold with at most terminal singularities
such that $K_X$ is pseudoeffective. 
We fix  $R := \R^+ [\Gamma_{i_0}]$ a $K_X$-negative extremal ray in $\NAX$. 
\end{assumption*}

\begin{definition}
We say that the $K_X-$negative extremal ray $R$ is small if every curve $C \subset X$ with
$[C] \in R$ is very rigid in the sense of Definition \ref{definitionveryrigid}.
Otherwise we say that the extremal ray $R$ is divisorial.
\end{definition}

\begin{remark} \label{remarksmall}
If the extremal ray $R$ admits a Mori contraction $\holom{\varphi}{X}{Y}$, the contraction
is small (resp. divisorial) if and only if this holds for the extremal ray. Indeed if the extremal ray is small, then
by Theorem \ref{theoremveryrigid} we have $-K_X \cdot C \leq 1$ for every curve such that $[C] \in R$.
In particular there are only finitely many cohomology classes  in $R$ that are represented by curves. 
Since for every class there are only finitely many
deformation families and by hypothesis every curve in the ray is very rigid, we see that there are only finitely many curves
$C \subset X$ such that $[C] \in R$. Thus $\varphi$ contracts only finitely many curves, it is small.
\end{remark}

\begin{proposition} \label{propositionperp} 
There exists a nef class $\alpha \in N^1(X)$ such that
$$ 
R = \{ z \in \overline{NA}(X) \ \vert \ \alpha \cdot z = 0\},
$$
and such that, using the notation of Theorem \ref{theoremNApreliminary}, the class $\alpha$ is strictly positive on
$$
\left( \NAX_{K_X \geq 0} + \sum_{i \in I, i \neq i_0 } \R^+ [\Gamma_i] \right) \setminus \{0\}.
$$ 
We call $\alpha$ a nef supporting class for the extremal ray $R$.
\end{proposition} 

\begin{proof} 
We set
$$ 
V := \NAX_{K_X \geq 0} + \sum_{i \in I, i \neq i_0 } \R^+ [\Gamma_i]. 
$$ 
By Lemma \ref{lemmaclosure} the cone $V$ is closed and by Theorem \ref{theoremNApreliminary} we have
$$
\NAX = V + \R^+ [\Gamma_{i_0}]
$$
By \cite[Lemma 6.7(d)]{Deb01}
there exists a linear form on $N_1(X)$ which vanishes on $R$ and is positive on
$V \setminus \{0\}$.  This linear form gives the class $\alpha$ by Proposition~\ref{dual}.
\end{proof}

\subsection{Divisorial rays}

The following proposition is a particular case of \cite[Thm.2]{AV84} and a generalisation of the well-known
theorem of Grauert  \cite{Gra62}.

\begin{proposition} \label{propositioncontractdivisor}
Let $X$ be a normal compact complex space, and let $S$ be a prime divisor that is $\Q$-Cartier
of Cartier index $m$.
Suppose that $S$ admits a morphism with connected fibres $\holom{f}{S}{B}$
such that $\sO_S(-mS)$ is $f$-ample. Then there exists a bimeromorphic
morphism $\varphi: X \to Y$ 
to a normal compact  complex space $Y$ such that
$\varphi|_S=f$ and $\varphi|_{X \setminus S}$ is an isomorphism onto $Y \setminus B$.
\end{proposition}

\begin{proof} Since $\sO_S(-mS)$ is $f$-ample there exists a multiple $m'$ of $m$ such that
we have
$$
R^1 f_*(\sO_S(-km'S)) = 0 $$
for all $k \in \N$.
Let $\sI \subset \sO_X$ be the ideal sheaf corresponding to the Cartier divisor $-m'D$, then 
$\sI/\sI^2|_S$ is an $f$-ample line bundle on $S$. 
Since we have
$$ 
R^1 f_*(\sI^k/\sI^{k+1}) = R^1 f_*(\sO_S(-km'S)) = 0
$$
for all $k \in \N$, the natural morphisms
$$ f_*(\sO_X/\sI^{k+1}) \to f_*(\sO_X/\sI^k) $$
are surjective for all $k \in \N$. Thus we can apply \cite[Thm.2]{AV84} to conclude.
\end{proof}

\begin{lemma} \label{lemmairreducible} 
Suppose that the extremal ray $R$ is divisorial.
We set 
$$ 
S = 
\bigcup_{C \subset X, [C] \in R} C.
$$
Then $S$ is an irreducible Moishezon surface and 
$ S \cdot C < 0$
for all curves $C$ with $[C] \in R$.
\end{lemma} 

\begin{proof}  
By definition the extremal ray $R$ contains a class $[C]$ with $C$ a curve that is not very rigid.
By Lemma \ref{lemmarationalrepresentative},a) we may suppose that $\dim_C \Chow{X}>0$.
By Lemma \ref{lemmabasic} there exists a unique surface $S_j$ from the divisorial
Zariski decomposition \eqref{Bdecomposition} such that $C$ and its deformations are contained in $S_j$,
moreover we have $S_j \cdot C<0$.
  
Let now $C'$ be any curve such that $[C'] \in R$. Then we have $[C']=\lambda [C]$ for some $\lambda \in \Q^+$, hence
$$
S_j \cdot C' = \lambda S_j \cdot C < 0.
$$
Thus we have $C' \subset S_j$ and $S=S_j$. The Moishezon property of $S$ is shown in Lemma~\ref{lemmasurfaces}. 
\end{proof} 

\begin{notation} {\rm 
Suppose that the extremal ray $R = \R^+ [\Gamma_{i_0}]$ is divisorial, and let $S$ be the surface from Lemma \ref{lemmairreducible}.
Let $\nu: \tilde S \to S \subset X$ be the normalisation; then $\nu^*(\alpha)$ is a nef class on $\tilde S$ and we may consider the nef reduction
$$ 
\tilde f: \tilde S \to \tilde B$$
with respect to $\nu^*(\alpha)$, cf. \cite[Thm.2.6]{8authors}. 
By Lemma \ref{lemmarationalrepresentative},a) there exists a curve $C \subset X$ such that $[C] \in R$ and $\dim_C \Chow{X}>0$.
Since $S \cdot C<0$ by Lemma \ref{lemmairreducible}, the
deformations  $(C_t)_{t \in T}$ of $C$ induce a
dominating family $(\tilde C_t)_{t \in T'}$ of $\tilde S$ such that $\nu^*(\alpha) \cdot \tilde C_t=0$. By definition of the nef reduction
this implies
$$ n(\alpha) := \dim \tilde B \in \{0,1\}.$$}
\end{notation}

\begin{corollary} \label{corollarypoint} 
Suppose that the extremal ray $R$ is divisorial and $n(\alpha) = 0$. Then the surface 
$S$ can be blown down to a point: there exists a bimeromorphic
morphism $\varphi: X \to Y$ 
to a normal compact  threefold $Y$ with $\dim \varphi(S) = 0$ such that $\varphi|_{X \setminus S}$ is an isomorphism onto $Y \setminus 0$.
\end{corollary} 

\begin{proof} Since $n(\alpha) = 0$, the class $\nu^* (\alpha)$ is trivial on $\tilde S$ \cite[Thm.2.6]{8authors}. 
Hence $ \nu^*(\alpha) \cdot \tilde C = 0$ for all curves $ \tilde C \subset \tilde S,$
hence $\alpha \cdot C = 0$ for all curves $C \subset S$. 
By Proposition \ref{propositionperp} the class of every curve $C \subset S$ belongs to $R$
and $\alpha|_S \equiv 0$. 

Let $m$ be the Cartier index of $S$. By Lemma \ref{lemmairreducible} the divisor $-mS$ 
is strictly positive on the extremal ray $R$. Since $\alpha$ is strictly positive on
$$
\left( \NAX_{K_X \geq 0} + \sum_{i \in I, i \neq i_0 } \R^+ [\Gamma_i] \right) \setminus \{0\},
$$
there exists an $\varepsilon > 0$ such $\alpha - \varepsilon [mS]$ is strictly positive 
on $\NAX \setminus \{0\}$. By Corollary \ref{corollarykaehlerform} there exists a K\"ahler form 
$\omega$ on $X$ such that 
$$
\alpha - \varepsilon [mS] = [\omega].
$$
Since $\alpha|_S \equiv 0$ we obtain that the class of the Cartier divisor $[-mS|_S]$ 
is represented by the K\"ahler form $\frac{1}{\varepsilon} \omega$.
Thus $\sO_S(-mS)$ is ample.  We conclude with Proposition \ref{propositioncontractdivisor}.
\end{proof} 

\begin{lemma} \label{lemmacontractioncurve}
Suppose that the extremal ray $R$ is divisorial and $n(\alpha) = 1$. 
Then there exists a fibration $\holom{f}{S}{B}$ onto a curve $B$ such that
a curve $C \subset S$ is contracted if and only if $[C] \in R$.
\end{lemma}

\begin{proof}
Let $\tilde C$ be a general fibre of the nef reduction $\tilde f: \tilde S \to \tilde B$ with respect to $\nu^*(\alpha)$, 
then we have $\nu^*(\alpha) \cdot \tilde C=0$. By definition of $\alpha$ this implies that 
for $C:=\nu(\tilde C)$ we have $[C] \in R$.
Thus we have $K_X \cdot C<0$ and the cycle space $\Chow{X}$ has positive dimension in $C$, so by Lemma \ref{lemmabasic} we have
$$
K_S \cdot C<K_X \cdot C<0.
$$ 
The normalisation $\nu$ being finite and $X$ having only finitely many singular points, we can choose $\tilde C$ such that
the following three conditions hold:
\begin{itemize}
\item $C \subset X_{\nons}$, 
\item $C \not\subset S_{\sing}$, and
\item $\tilde C \subset \tilde S_{\nons}$,
\end{itemize}
By the first property the intersection numbers $K_X \cdot C<0$ and $S \cdot C<0$ are integers, hence $K_S \cdot C \leq -2$. 
By the second property and \eqref{inequalitycanonical} we have
$$
K_{\tilde S} \cdot \tilde C \leq K_S \cdot C \leq -2
$$
with equality if and only if $C \subset S_{\nons}$ and $K_X \cdot C=-1$ and $S \cdot C=-1$.
Since $\tilde C$ is a fibre of the fibration $\tilde f$ contained in the smooth locus of $\tilde S$, 
we have
$$ 
-2 \leq \deg K_{\tilde C} = K_{\tilde S} \cdot \tilde C \leq -2.
$$
Thus $\tilde C=C$ is a rational curve contained in $S_{\nons}$ such that $K_S \cdot C=-2$. 
In particular the sheaf $\sO_S(C)$ is invertible. 
Consider the exact sequence
$$
0 \rightarrow \sO_S(mC) \rightarrow \nu_* \sO_{\tilde S}(mC) \rightarrow \sF \otimes  \sO_{\tilde S}(mC) \rightarrow 0,
$$ 
where $\sF$ is a coherent sheaf supported on the non-normal locus of $S$. 
Since $C \subset S_{\nons}$, we have  $\sF \otimes  \sO_{\tilde S}(mC) \simeq \sF$ for all $m \in \N$.
Since $h^0(\tilde S, \sO_{\tilde S}(mC))$ goes to infinity
as $m \rightarrow \infty$ we see that $\kappa(S, \sO_S(C))=1$. Since $C^2=0$ we see that some multiple of $C$ defines a fibration
$\holom{f}{S}{B}$ onto a curve $B$ such that we have a commutative diagram
$$
\xymatrix{
\tilde S  \ar[d]_{\tilde f} \ar[r]^{\nu} & S  \ar[d]^{f}
\\
\tilde B \ar[r] & B 
}
$$
\end{proof}

\begin{corollary} \label{corollarycurve}
Suppose that the extremal ray $R$ is divisorial and $n(\alpha) = 1$.
Let $\holom{f}{S}{B}$ be the fibration defined in Lemma \ref{lemmacontractioncurve}.
Then the surface $S$ can be contracted onto a curve:
there exists a bimeromorphic
morphism $\varphi: X \to Y$ 
to a normal compact threefold $Y$ such that
$\varphi|_S=f$ and $\varphi|_{X \setminus S}$ is an isomorphism onto $Y \setminus B$.
\end{corollary} 

\begin{proof} 
Let $m$ be the Cartier index of $S$. By Lemma \ref{lemmacontractioncurve} the surface $S$
admits a fibration $f$ onto a curve contracting exactly the curves in $R$. By Lemma \ref{lemmairreducible}
the divisor $\sO_X(-mS)$ is ample on every curve in $R$, so its restriction $\sO_S(-mS)$ is $f$-ample.
Conclude with Proposition \ref{propositioncontractdivisor}.
\end{proof} 

\subsection{Small rays}

\begin{notation} \label{notationsmalllocus}{\rm
Suppose that the extremal ray $R = \R^+ [\Gamma_{i_0}]$ is small.
Set 
$$
C := \cup_{C_l \subset X, [C_l] \in R} C_l,
$$
then $C$ is a finite union of curves by Remark \ref{remarksmall}. We say that $C$ is contractible
if there exists a bimeromorphic
morphism $\varphi: X \to Y$ 
onto a normal threefold $Y$ with $\dim \varphi(C) = 0$ such that $\varphi|_{X \setminus C}$ is an isomorphism onto 
$Y \setminus \varphi(C)$. 
}
\end{notation}

\begin{proposition} \label{propositionalphabig}
Suppose that the extremal ray $R=\R^+ [\Gamma_i]$ is small, and let $\alpha$ be the nef supporting class
of $R$ defined in Proposition \ref{propositionperp}. 
Let $S \subset X$ be an irreducible surface. Then we have
$\alpha^2 \cdot S>0$.
\end{proposition} 

\begin{proof}[Proof of Proposition \ref{propositionalphabig}]
Since $\alpha$ is nef we have $\alpha^2 \cdot S \geq 0$. Arguing by contradiction we suppose that $\alpha^2 \cdot S = 0$.

By hypothesis the cohomology class $\alpha-K_X$ is positive on the extremal ray $R$, moreover we know 
by Proposition \ref{propositionperp} that $\alpha$ is positive on
$$
\left( \NAX_{K_X \geq 0} + \sum_{i \in I, i \neq i_0 } \R^+ [\Gamma_i] \right) \setminus \{0\}.
$$ 
Thus, up to replacing $\alpha$ by some positive multiple, we can suppose that $\alpha-K_X$ is
positive on $\NAX \setminus \{ 0 \}$. Since $X$ is a K\"ahler space this implies by Corollary \ref{corollarykaehlerform}
that $\alpha-K_X$ is a K\"ahler class $\omega$.

If $\alpha|_S=0$ we obtain
$-K_X|_S = \omega|_S$.
Thus the divisor $-K_X|_S$ is ample, in particular $S$
is projective and covered by curves. Since $\alpha|_S=0$, the classes of all these curves are in the small ray $R$,
a contradiction. 

Suppose now that $\alpha|_S \neq 0$. Then we have
$$
0 = \alpha^2 \cdot S = K_X \cdot \alpha \cdot S
+ \omega \cdot \alpha \cdot S 
$$ 
and 
$$
\omega \cdot \alpha \cdot S = \omega|_S \cdot \alpha|_S >0
$$ 
by the Hodge index theorem. 
Thus we obtain
\begin{equation} \label{help1}
K_X \cdot \alpha \cdot S < 0.
\end{equation}
In particular $K_X|_S$ is not pseudoeffective, the class $\alpha|_S$ being nef. Using the decomposition
\eqref{Bdecomposition} we deduce immediately that $S \cdot \alpha \cdot S < 0$.
In particular we have
\begin{equation} \label{help2}
(K_X+S) \cdot \alpha \cdot S<0.
\end{equation}
Let \holom{\pi}{S'}{S} be the composition of normalisation and minimal resolution (cf. Subsection \ref{subsectionsurfaces}). 
By \eqref{equationresolve} we have $K_{S'}+E= \pi^* K_S$ with $E$ an effective divisor. 
Since $E$ is effective and $\alpha$ is nef, the inequality \eqref{help2} now implies
\begin{equation} \label{help}
K_{S'} \cdot \pi^* \alpha|_S \leq K_S \cdot \alpha|_S < 0.
\end{equation}
Thus $K_{S'}$ is not pseudoeffective, hence $S'$ is uniruled and $H^2(S', \sO_{S'})=0$. 
In particular we can see the nef class $\pi^* \alpha$ as an $\R$-divisor on the projective surface $S'$. 
Since $\pi^* \alpha$ is nef and $S'$ is a surface, it is an element of
$\overline{\mbox{NM}}(S')$ the closure of the cone of movable curves. 
The extremal ray $R$ contains only the classes of finitely many curves, so $\pi^* \alpha$ is strictly
positive on every movable curve in $S'$.

Fix an ample divisor $A$ on $S'$. 
By \cite[Thm.1.3]{Ara10} for every $\varepsilon>0$ we have a decomposition
$$
\pi^* \alpha = C_{\varepsilon} + \sum \lambda_{i, \varepsilon} M_{i, \varepsilon}
$$
where $\lambda_{i, \varepsilon} \geq 0$, $(K_{S'}+\varepsilon A) \cdot C_\varepsilon \geq 0$ and the $M_{i, \varepsilon}$ are movable curves. Since $(\pi^* \alpha)^2=0$ and $\pi^* \alpha \cdot M_{i, \varepsilon}>0$ we see that
$\pi^* \alpha = C_{\varepsilon}$
for all $\varepsilon>0$. Passing to the limit we obtain $K_{S'} \cdot \pi^* \alpha \geq 0$, a contradiction to \eqref{help}.
\end{proof}

\begin{theorem} \label{theoremsmallray}
Suppose that the extremal ray $R=\R^+ [\Gamma_i]$ is small, and let
$C$ be the locus covered by curves in $R$ (cf. Notation \ref{notationsmalllocus}).
Then $C$ is contractible. 
\end{theorem} 

\begin{proof}[Proof of Theorem \ref{theoremsmallray}] 
Let $\alpha  \in N^1(X)$ be as in Proposition~\ref{propositionperp}. 
By definition of the class $\alpha$, the class  $-K_X+\alpha$ is positive on the extremal ray $R$.
Since $\alpha$ is strictly positive on 
$$
\left( \NAX_{K_X \geq 0} + \sum_{i \in I, i \neq i_0 } \R^+ [\Gamma_i] \right) \setminus \{0\},
$$ 
we can suppose, up to replacing $\alpha$ by some positive multiple, that $-K_X+\alpha$ 
is strictly positive on this cone. Hence $-K_X+\alpha$ is strictly positive on $\NAX$.
Thus we know by Corollary \ref{corollarykaehlerform} that $-K_X+\alpha$ is a K\"ahler class. Since $\alpha$ is the sum of a pseudoeffective class and a K\"ahler class, 
it is big (i.e. lies in the interior of the cone generated by pseudoeffective classes). 
Since $\alpha$ is nef this implies that $\alpha^3>0$.

Let $\pi: \hat X \to X$ be a desingularisation. 
Since $\alpha$ is nef and big, the pull-back
$\pi^* \alpha$ is nef and big. By \cite[Thm.1.1]{CT13} the
non-K\"ahler locus $E_{nK}(\pi^* \alpha)$ is equal to the null-locus, i.e. we have
$$
E_{nK}(\pi^* \alpha) = \cup_{(\pi^* \alpha)|_Z^{\dim Z}=0} Z, 
$$
where the union runs over all the subvarieties of $\hat X$.
If $Z \subset X$ is a surface such that $(\pi^* \alpha)|_Z^{2}=0$ then it follows from the projection formula
and Proposition \ref{propositionalphabig} that $\dim \pi(Z) \leq 1$.
Thus we see that $E_{nK}(\pi^* \alpha)$
is a finite union of $\pi$-exceptional surfaces and curves.
Since
$$
E_{nK}(\alpha) \subset \pi(E_{nK}(\pi^* \alpha)),
$$
and $X_{\sing}$ consists of finitely many points,
we finally obtain that 
$$
E_{nK}(\alpha) = \cup_{\alpha \cdot B=0} B = C.
$$

Recall that a cohomology class is modified K\"ahler \cite[Defn.2.2]{Bou04} if it contains a K\"ahler current
$T$ such that for every prime divisor $D \subset X$ the Lelong number in the generic point of $D$ is equal to zero.
Since the non-K\"ahler locus of $\alpha$ does not contain any divisor, it is a modified K\"ahler class. Thus by 
\cite[Prop.2.3]{Bou04}\footnote{Proposition 2.3. in \cite{Bou04} is for compact complex manifolds, but the proof goes
through without changes for a variety with isolated singularities. A different way to obtain the decomposition is to 
take a resolution of singularities \holom{\nu}{X'}{X} and consider the nef and big class $\nu^* 
\alpha$. The non-K\"ahler locus is the union of the $\nu$-exceptional locus and the strict transforms of the curves
in the non-K\"ahler locus of $\alpha$. By \cite[Thm.3.1.24]{Bou02} there exists a modification
$\holom{\mu'}{\tilde X}{X'}$ and a K\"ahler class $\tilde \alpha$ on $\tilde X$ such that $(\mu')^* \nu^* \alpha= \tilde \alpha + E$.
Then $\mu:=\nu \circ \mu'$ has the stated properties.} there exists a modification
$\holom{\mu}{\tilde X}{X}$ and a K\"ahler class $\tilde \alpha$ on $\tilde X$ such that $\mu_* \tilde \alpha= \alpha$.
Since $\tilde \alpha - \mu^* \alpha$ is $\mu$-nef, we obtain
\begin{equation} \label{equationbig}
\mu^* \alpha = \tilde \alpha + E  
\end{equation}
with $E$ a $\mu$-exceptional $\R$-divisor such that
$$
\mu(\supp E)= E_{nK}(\alpha). 
$$
Since $\mu_*(E)=0$ and $-E$ is $\mu$-nef, the divisor $E$ is effective:
the usual proof of the negativity lemma \cite[Lemma 3.39]{KM98} \cite[3.6.2]{BCHM10} works since the statement is local on the
base and $\mu$ is a projective morphism.
Note also that since $-E$ is $\mu$-ample and antieffective, its support is equal to the exceptional locus of $\mu$.

Let now $Z \subset X$ be a connected component of the curve $C$, and set
$D_Z$ for the support of $\fibre{\mu}{Z}$.
Then we have $\alpha|_{Z} \equiv 0$, so by \eqref{equationbig} we see that
$$
-E|_{D_Z} = \tilde \alpha|_{D_Z}
$$
is ample. Since ampleness is an open property we can slightly perturb the coefficients of $E$ to obtain an effective $\Q$-divisor $E' \subset \tilde X$ such that $-E'|_{D_Z}$ is ample.
Up to taking a positive multiple we can suppose that $E'$ has integer coefficients.
Since $Z$ is a connected component of $E_{nK}(\alpha)$, the divisor $D_Z$ is a connected component of
the support of $E'$. Thus we can write $E' = \sum m_i D_i + R$ where $D_i \subset D_Z$ is an irreducible component and $R$ is disjoint from $D_Z$. Since $-E|_{D_Z} = - (\sum m_i D_i)_{D_Z}$ is ample and a line bundle on a compact complex space is ample if and only if it is ample on the reduction, we finally see that
$$
N^*_{\sum m_i D_i/\hat X} = \sO_{\sum m_i D_i}(-\sum m_i D_i)
$$
is ample. By Grauert's criterion \cite{Gra62} there exists a morphism $\holom{\psi}{\hat X}{Y_Z}$ contracting
the divisor $\sum m_i D_i$ onto a point. By the rigidity lemma \cite[Lemma 4.1.13]{BS95} there exists a morphism
\holom{\varphi_Z}{X}{Y_Z} such that $\psi = \varphi_Z \circ \mu$. The contraction $\varphi$ is obtained by
applying this construction to all the connected components of $C$.
\end{proof}

\section{Running the MMP- Proof of Theorems \ref{theoremminimalmodel}  and \ref{theoremNA}}

The following statement is well-known to specialists, we include the proof for the convenience of the reader.

\begin{proposition} \label{propositioncontraction}
Let $X$ be a  normal $\Q$-factorial  compact K\"ahler space with at most terminal singularities.
Let $\R^+ [\Gamma_i]$ be a $K_X$-negative extremal ray in $\NAX$.
Suppose that there exists a morphism \holom{\varphi}{X}{Y} onto a normal complex space $Y$
such that $-K_X$ is $\varphi$-ample and
a curve $C \subset X$ is contracted if and only if $[C] \in \R^+ [\Gamma_i]$.
\begin{enumerate}
\item We have exact sequences
\begin{equation} \label{exactH2}
0 \rightarrow H^2(Y, \R) \stackrel{\varphi^*}{\rightarrow} H^2(X, \R) \stackrel{[L] \mapsto L \cdot \Gamma_i}{\longrightarrow} \R \rightarrow 0
\end{equation}
and
\begin{equation} \label{exactN1}
0 \to N^1(Y)  \stackrel{\varphi^*}{\rightarrow}  N^1(X) \stackrel{[L] \mapsto L \cdot \Gamma_i}{\longrightarrow} \R \to 0.
\end{equation}
In particular we have $b_2(X) = b_2(Y) + 1$. 
\item We have an exact sequence
\begin{equation} \label{exactPic}
0 \rightarrow \pic(Y) \stackrel{\varphi^*}{\rightarrow} \pic(X) \stackrel{[L] \mapsto L \cdot \Gamma_i}{\longrightarrow} \Z. 
\end{equation}
\item If the contraction is divisorial, the variety $Y$ has at most terminal $\Q$-factorial singularities and its Picard number 
is $\rho(X)-1$. 
\item If the contraction is small with flip $X^+ \rightarrow Y$, the variety $X^+$ 
has at most terminal $\Q$-factorial singularities and its Picard number 
is $\rho(X)$. 
\end{enumerate}
\end{proposition}

\begin{proof}
Note first that the morphism $\varphi$ is projective since $-K_X$ is $\varphi$-ample.
Moreover by \cite[1-2-5]{KMM87}, \cite{Nak87} we have $R^j \varphi_* \sO_X=0$ for all $j \geq 1$.
In particular $Y$ has at most rational singularities.
Statement a) is thus a special case of Lemma \ref{pullback}.

{\em Proof of statement b).}
The non-trivial part of this statement is to prove that if $L$ is a line bundle on $X$ such that $L \cdot \Gamma_i=0$, 
there exists a line bundle $L'$ on $Y$ such that $L \simeq \varphi^* L'$. 
Note that if such a $L'$ exists, we have
$$
L' \simeq \varphi_* L.
$$
Thus $L'$ is unique up to isomorphism and it is sufficient to prove that the direct image sheaf $\varphi_* L$ is locally free.
This is a local property, so fix an arbitrary point $y \in Y$ and a small Stein neighbourhood
$y \in U \subset Y$. Set $X_U:=\fibre{f}{U}$, then the morphism $\holom{f:=\varphi|_{X_U}}{X_U}{U}$ satisfies the conditions
of the relative base point free theorem \cite[Thm.3.3]{Anc87}. Thus $L^{\otimes b}|_{X_U}$ is $f$-globally generated for {\em all} $b \gg 0$. The base $U$ being Stein 
we see that $L^{\otimes b}|_{X_U}$ is globally generated for {\em all} $b \gg 0$. 
Since $L$ is $f$-numerically trivial we have 
$$
L^{\otimes b}|_{X_U} \simeq f^* M, \ L^{\otimes b+1}|_{X_U} \simeq f^* N
$$
for some line bundles $M$ and $N$ on $U$. In particular we have $f_*(L|_{X_U}) \simeq N \otimes M^*$. 

{\em Proof of statements c) and d).}
These are a well-known consequences of statement b), compare the proof of \cite[Prop.7.44]{Deb01}. 
\end{proof}

\begin{corollary} \label{corollarycontractionkaehler}
Let $X$ be a  normal  $\Q$-factorial compact K\"ahler threefold with at most terminal singularities.
Let $\R^+ [\Gamma_i]$ be a $K_X$-negative extremal ray in $\NAX$.
Suppose that there exists a bimeromorphic morphism \holom{\varphi}{X}{Y} 
such that $-K_X$ is $\varphi$-ample and
a curve $C \subset X$ is contracted if and only if $[C] \in \R^+ [\Gamma_i]$.
Then $Y$ is a K\"ahler space. 
\end{corollary} 

\begin{proof} 
Let $\alpha$ be a nef supporting class (cf. Proposition \ref{propositionperp}) for the extremal ray $\R^+ [\Gamma_i]$. 
Then $\alpha \cdot \Gamma_i = 0$, so by Proposition \ref{propositioncontraction}a) we have 
$\alpha  = \varphi^*(\alpha')$ with some $\alpha' \in N^1(Y)$. 
Since $\alpha \cdot \beta > 0$ for all $\beta \in \NAX \setminus \R^+ [\Gamma_i]$ and the map
$\varphi_*: \NAX \rightarrow \overline{NA}(Y)$ is surjective by Proposition \ref{ext}, we see that
$\alpha'$ is strictly positive on the closed cone $\overline{NA}(Y)$.
Hence $\alpha'$ is a K\"ahler class by Theorem~\ref{kaehler2}, once we have verified that $Y$ does not contain an irreducible
curve $C$ homologous to $0.$ This is clear if the extremal ray $\R^+ [\Gamma_i]$ is small or $\varphi$ contracts a divisor
to a point. Suppose now that $\varphi$ contracts a divisor $E$ to a curve $C' \subset Y$; this curve being the only candidate for a curve being homologous to $0.$ 
Since the morphism $\varphi|_E: E \rightarrow C'$ is
projective, so does $E$, and thus there exists a curve $C \subset E \subset X$ such that $\varphi(C)=C'$. Thus for some positive number $d$ we have 
$$
\alpha' \cdot (d C') = \alpha' \cdot \varphi_*(C) =  \alpha \cdot C > 0,
$$
since $0 \neq [C] \in \NAX$ and $[C] \not\in \R^+ [\Gamma_i]$.
Thus we have $[C'] \neq 0$.     
\end{proof}

\begin{proof}[Proof of Theorem  \ref{theoremcontraction}]
The existence of a morphism \holom{\varphi}{X}{Y} contracting exactly the curves in the extremal ray 
is established in the Corollaries \ref{corollarypoint}, \ref{corollarycurve} and Theorem \ref{theoremsmallray}.
Moreover $Y$ is a K\"ahler space by Corollary \ref{corollarycontractionkaehler}.
\end{proof}

\begin{proof}[Proof of Theorem \ref{theoremminimalmodel}]

{\em Step 1: Running the MMP.} If $K_X$ is nef, we are finished. 
If $K_X$ is not nef, there exists by Theorem \ref{theoremNA} a $K_X$-negative extremal ray $R$ in $\NAX$.
By Theorem \ref{theoremcontraction} the contraction  \holom{\varphi}{X}{Y} of $R$ exists.
If $R$ is divisorial we can continue the MMP with $Y$ by Proposition \ref{propositioncontraction},c).  
If $R$ is small, we know by Mori's flip theorem \cite[Thm.0.4.1]{Mor88} that the flip $\holom{\varphi^+}{X^+}{Y}$ exists, and
by Proposition \ref{propositioncontraction},d) we can continue the MMP with $X^+$. 

{\em Step 2: Termination of the MMP.} Recall that for a normal compact threefold $X$ with at most terminal singularities, 
the difficulty $d(X)$ \cite{Sho85} is defined by
$$
d(X) := \# \{ i \ | \ a_i < 1 \},
$$
where $K_Y = \mu^* K_X + \sum a_i E_i$ and $\holom{\mu}{Y}{X}$ is any resolution of singularities.
By \cite[Lemma 5.1.16]{KMM87}, \cite{Sho85} we have $d(X)>d(X^+)$, if $X^+$ is the flip of a small contraction.
Since the Picard number and the difficulty are non-negative integers, any MMP terminates after finitely many steps.
\end{proof}

\newcommand{\etalchar}[1]{$^{#1}$}
\def\cprime{$'$}

\end{document}